\documentclass[runningheads]{llncs}
\usepackage{graphicx}
\usepackage{hyperref}
\usepackage[utf8]{inputenc}
\usepackage[english]{babel}
\usepackage{amssymb,amsmath}
\usepackage{etoolbox}
\usepackage{enumitem}
\usepackage{graphicx}
\usepackage{tikz}
\usepackage{dsfont}
\usepackage{soul}
\usepackage{mathtools}

\newcommand{\bc}{\begin{center}}

\newcommand{\bea}{\begin{align}}
\newcommand{\beq}{\begin{equation}}
\newcommand{\beqa}{\begin{eqnarray}}
\newcommand{\bi}{\begin{itemize}}
\newcommand{\bitem}{\begin{itemize}}
\newcommand{\bpm}{\begin{pmatrix}}      
\newcommand{\BV}{\tmop{BV}}

\newcommand{\cref}[1]{ {\tiny[{#1}]}}

\newcommand{\ddiv}{\tmop{div}}
\newcommand{\supp}{\tmop{supp}}

\newcommand{\ec}{\end{center}}
\newcommand{\eea}{\end{align}}
\newcommand{\eeq}{\end{equation}}
\newcommand{\eeqa}{\end{eqnarray}}
\newcommand{\ei}{\end{itemize}}
\newcommand{\eitem}{\end{itemize}}

\newcommand{\epm}{\end{pmatrix}}

\newcommand{\N}{\mathbb{N}}

\newcommand{\R}{\mathbb{R}}

\newcommand{\tm}[1]{}

\newcommand{\tmop}[1]{\ensuremath{\operatorname{#1}}}

\newcommand{\TV}{\tmop{TV}}

\newcommand{\bpf}{\begin{proof}}
\newcommand{\epf}{\end{proof}}

\newcommand{\be}{\begin{enumerate}}
\newcommand{\ee}{\end{enumerate}}

\usepackage{soul}

\definecolor{oceangreen}{RGB}{0,57,74}
\definecolor{oceangreen-90}{RGB}{0,73,90}
\definecolor{oceangreen-80}{RGB}{0,90,107}
\definecolor{oceangreen-70}{RGB}{17,107,123}
\definecolor{oceangreen-60}{RGB}{60,126,142}
\definecolor{oceangreen-50}{RGB}{93,146,160}
\definecolor{oceangreen-40}{RGB}{126,168,178}
\definecolor{oceangreen-30}{RGB}{157,188,198}
\definecolor{oceangreen-20}{RGB}{189,210,216}
\definecolor{oceangreen-10}{RGB}{222,233,236}
\definecolor{red-lightest}{RGB}{233,191,173}
\definecolor{red-light}{RGB}{227,32,50}
\definecolor{red-medium}{RGB}{182,22,32}
\definecolor{red-dark}{RGB}{126,20,24}
\definecolor{orange-lightest}{RGB}{240,205,178}
\definecolor{orange-light}{RGB}{236,117,4}
\definecolor{orange-medium}{RGB}{202,81,25}
\definecolor{orange-dark}{RGB}{129,53,18}
\definecolor{yellow-lightest}{RGB}{236,217,186}
\definecolor{yellow-light}{RGB}{251,186,0}
\definecolor{yellow-medium}{RGB}{191,134,21}
\definecolor{yellow-dark}{RGB}{117,83,17}
\definecolor{green-lightest}{RGB}{219,215,187}
\definecolor{green-light}{RGB}{149,187,12}
\definecolor{green-medium}{RGB}{126,133,37}
\definecolor{green-dark}{RGB}{50,89,74}
\definecolor{turquoise-lightest}{RGB}{205,219,216}
\definecolor{turquoise-light}{RGB}{59,178,160}
\definecolor{turquoise-medium}{RGB}{36,143,133}
\definecolor{turquoise-dark}{RGB}{0,90,91}
\definecolor{ocean-lightest}{RGB}{198,220,226}
\definecolor{ocean-light}{RGB}{0,174,198}
\definecolor{ocean-medium}{RGB}{0,145,167}
\definecolor{ocean-dark}{RGB}{0,97,122}
\definecolor{cyan-lightest}{RGB}{195,217,229}
\definecolor{cyan-light}{RGB}{60,169,213}
\definecolor{cyan-medium}{RGB}{0,131,173}
\definecolor{cyan-dark}{RGB}{0,87,119}
\definecolor{blue-lightest}{RGB}{195,217,236}
\definecolor{blue-light}{RGB}{111,165,206}
\definecolor{blue-medium}{RGB}{0,105,163}
\definecolor{blue-dark}{RGB}{0,70,114}
\definecolor{gray-cold-lightest}{RGB}{215,223,228}
\definecolor{gray-cold-light}{RGB}{191,203,213}
\definecolor{gray-cold-medium}{RGB}{140,157,171}
\definecolor{gray-cold-dark}{RGB}{87,105,120}
\definecolor{gray-neutral-lightest}{RGB}{226,228,228}
\definecolor{gray-neutral-light}{RGB}{208,208,210}
\definecolor{gray-neutral-medium}{RGB}{156,157,160}
\definecolor{gray-neutral-dark}{RGB}{100,100,102}
\definecolor{gray-warm-lightest}{RGB}{237,236,227}
\definecolor{gray-warm-light}{RGB}{218,214,203}
\definecolor{gray-warm-medium}{RGB}{162,159,145}
\definecolor{gray-warm-dark}{RGB}{118,115,105}

\usepackage{pgfplots}
\pgfplotsset{compat=1.13}
\usepgfplotslibrary{fillbetween}
\pgfkeys{/pgf/declare function={conv(\a,\b,\s) = (1 - \s)*\a + \s*\b;}}

\def\nsamplesd{50}

\def\xlima{0.1}
\def\xlimb{1.2}
\def\zlima{0.1}
\def\zlimb{0.9}
\pgfkeys{/pgf/declare function={zcoord(\s) = conv(\zlima,\zlimb,\s);}}
\pgfkeys{/pgf/declare function={xcoord(\s) = conv(\xlima,\xlimb,\s);}}

\def\fxmax{-0.25}
\def\fxmin{0.25}
\pgfkeys{/pgf/declare function={f0(\x) =   \x^3 - 3*(\fxmax + \fxmin)/2*\x^2 + 3*\fxmax*\fxmin*\x + 0.2;}}
\pgfkeys{/pgf/declare function={f1(\x) = 3*\x^2 - 3*(\fxmax + \fxmin)  *\x   + 3*\fxmax*\fxmin;}}
\pgfkeys{/pgf/declare function={f(\x) = 2*f0(\x - 0.6);}}
\pgfkeys{/pgf/declare function={ff(\x) = 2*f1(\x - 0.6);}}

\pgfkeys{/pgf/declare function={rho0(\x) = 0.7*exp(-0.5*\x^2);}}
\pgfkeys{/pgf/declare function={rho2(\t) = conv(10,12,\t));}}
\pgfkeys{/pgf/declare function={rho3(\t) = f(xcoord(\t)));}}
\pgfkeys{/pgf/declare function={rho(\t,\z) = rho2(\t)*rho0(rho2(\t)*(\z - rho3(\t)));}}

\usetikzlibrary{positioning}
\usetikzlibrary{calc}
\usetikzlibrary{spy}

\newlength\picwidth
\newlength\spywidth
\newlength\crvogt

\renewcommand{\d}{\,\mathrm{d}}

\newtheorem{thm}{Theorem}
\newtheorem{defn}[thm]{Definition}
\newtheorem{prop}[thm]{Proposition}
\newtheorem{lem}[thm]{Lemma}
\newtheorem{cor}[thm]{Corollary}

\begin{document}

\title{On the Connection between Dynamical Optimal Transport and Functional Lifting} %
\titlerunning{Connections between Dynamical Optimal Transport and Functional Lifting}
%
\author{Thomas Vogt\inst{1} \and
Roland Haase\inst{2} \and
Danielle Bednarski\inst{2} \and
Jan Lellmann\inst{2}}
\authorrunning{T.~Vogt et al.}
%
\institute{Potsdam Institute for Climate Impact Research, 14473 Potsdam, Germany\\ \email{thomas.vogt@pik-potsdam.de}\and
University of Lübeck, 23562 L\"ubeck, Germany\\
\email{\{haase,bednarski,lellmann\}@mic.uni-luebeck.de}}
\maketitle              
\begin{abstract}
Functional lifting methods provide a tool for approximating solutions of difficult non-convex problems by embedding them into a larger space. In this work, we investigate a mathematically rigorous formulation based on  embedding into the space of pointwise probability measures over a fixed range~$\Gamma$.
Interestingly, this approach can be derived as a generalization of the theory of dynamical optimal transport.
Imposing the established continuity equation as a constraint corresponds to variational models with first-order regularization. By modifying the continuity equation, the approach can also be extended to models with higher-order regularization.

\end{abstract}

\setcounter{footnote}{0}


\section{Motivation and Introduction}
\label{sec:intro}
Over the last two decades, functional lifting techniques have been established as a powerful and versatile tool for solving variational problems in image processing.
While originally concerned \cite{alberti2003,pock2008,pock2010} with convex relaxations of functionals of the form
\begin{equation}
        F(u) := \int_{\Omega} \rho(x, u(x)) + \eta(\nabla u(x)) \d x,
        \label{eq:calibration_functional}
\end{equation}
in which $\Omega \subset \R^d$ is open and bounded, $\Gamma \subset \R$ is compact, $\rho: \Omega \times \Gamma \to \R$ is a pointwise data term and $\eta: \R^d \to \R$ is a convex regularizer, numerous extensions of the concept have been proposed in recent years.
Among these are lifting approaches for vector-valued \cite{goldluecke2013,strekalovskiy2014} and manifold-valued problems \cite{lellmann2013b,vogt2019b}, i.e., for $\Gamma \subset \R^s$ and $\Gamma = \mathfrak{M}$ for manifolds $\mathfrak{M} \subset \R^s$, as well as functionals with higher-order regularization that involve, e.g., the Laplacian \cite{loewenhauser2018,vogt2019a} or the total generalized variation \cite{ranftl2013,strecke2019} of $u$.
Further extensions include a specially tailored discretization technique known as \textit{sublabel-accurate} liftings \cite{mollenhoff2016,laude2016,mollenhoff2017} and a generalization to polyconvex regularizers of first order \cite{mollenhoff2019}.

Although these generalizations achieve promising results in practice, many of them lack a theoretically sound continuous formulation or, respectively, one that establishes a connection to the original \textit{calibration method} for~\eqref{eq:calibration_functional} from \cite{pock2010}.
In some cases this is due to an early discretization of the range $\Gamma$ in the derivation of the model, e.g., in \cite{laude2016,loewenhauser2018}, whereas others rely on heuristic descriptions of the continuous case without a proper discussion of the correct function spaces such as \cite{goldluecke2013,vogt2019a}.

In this work, we strive for a fully continuous lifting model that encompasses a large number of the aforementioned approaches, while at the same time providing the desired theoretical soundness.
To be exact, we consider functionals of the form
\begin{equation}
        F(u) := \int_{\Omega} f(x, u(x), Lu(x)) \d x,
        \label{eq:target_functional}
\end{equation}
where $\Omega \subset \R^d$ is open and bounded, $\Gamma \subset \R^s$ is a compact \textit{vectorial} range and $L \in \{ \nabla, \nabla^2, \Delta \}$ is one of the listed linear differential operators.
The integrand $f: \Omega \times \Gamma \times \R^m \to \R$ is expected to be \textit{convex in the third argument} and bounded from below.
It is easy to see how \eqref{eq:target_functional} generalizes the calibration method functional~\eqref{eq:calibration_functional}.

We propose a convex relaxation strategy for \eqref{eq:target_functional} based on the use of measures over $\Omega \times \Gamma$ as lifted variables.
More precisely, we define a lifted convex functional~$\mathcal{F}$ on the space $\mathcal{M}(\Omega \times \Gamma)$ of Radon measures such that
$\mathcal{F}(\delta_u)$ provably agrees with $F(u)$ for all sufficiently regular functions $u: \Omega \to \Gamma$ (as required by the differential operator~$L$).
The notation $\delta_u$ refers to a measure concentrated on the graph of~$u$.

The foundation for our definition of $\mathcal{F}$ is borrowed from the theory of \textit{dynamical optimal transport} as we introduce a highly generalized version of the \textit{Benamou-Brenier functional} from \cite{santambrogio2015} in our derivations.
Together with the results from \cite{perkkio2018}, we prove our main theoretical result, namely an integral representation of the implicitly defined Benamou-Brenier functional.
Put in broader terms, the original notion of dynamical optimal transport of assigning a kinetic-energy cost to a time-dependent family of mass distributions (cf. \cite{benamou2000,brenier2003}) is generalized to the application of a largely arbitrary cost-rule to a family of mass distributions which may be indexed by multiple variables\footnote{In our case, these variables are the coordinates of $\Omega$.}.


On the application side, we obtain a fully convex model by relaxing the nonconvex domain of $\mathcal{F}$, i.e., the set of all graph-concentrated measures $\delta_u$ for sufficiently regular $u$, to ${\{\mu : \Omega \to \mathcal{P}(\Gamma) \}}$, the convex set of all functions that assign probability measures over $\Gamma$ to points in $\Omega$.

\subsection{Related Work}
\label{subsec:related_work}

\subsubsection{Functional Lifting}
A pioneering approach for convexifying scalar-valued problems of the type \eqref{eq:calibration_functional} from \cite{pock2008,pock2010} is based on subgraph-representations of functions~${u \in W^{1,1}(\Omega)}$.
It starts by associating $u$ with a ``lifted'' representation through its subgraph $\mathbf{1}_u : \Omega \times \Gamma \to \{0, 1\}$ defined by
\begin{equation}
        \mathbf{1}_u(x, z) := \begin{cases}
                1, & u(x) > z,\\
                0, & \text{otherwise}.
        \end{cases}
        \label{eq:subgraph_fctn}
\end{equation}
Then a lifted convex functional $\mathcal{F}$ is constructed,\begin{equation}
        \mathcal{F}(v) := \sup_{\phi \in \mathcal{K}} \int_{\Omega \times \Gamma} \langle \phi , Dv \rangle ,
        \label{eq:calibration_lifting}
\end{equation}
with $Dv$ being the distributional derivative of $v \in \BV(\Omega \times \Gamma)$ and $\mathcal{K}$ the set of admissible dual vector fields
\begin{multline}
        \mathcal{K} := \{ (\phi^x, \phi^z) \in C_0(\Omega \times \Gamma, \R^d \times \R) :\\ \eta^*(\phi^x(x, z)) - \rho(x, z) \leq \phi^z(x, z) \ \forall (x, z) \in \Omega \times \Gamma \}.
        \label{eq:calibration_dual_set}
\end{multline}
With these definitions, one can show the equivalence $F(u) = \mathcal{F}(\mathbf{1}_u)$ \cite[Thm. 3.2]{pock2010}. Thus the problem of minimizing $F$ over $W^{1,1}(\Omega)$ can equivalently be formulated as minimizing $\mathcal{F}$ over the (nonconvex) set ${\{ \mathbf{1}_u : u \in W^{1,1}(\Omega) \}}$.

A striking -- and exclusive to the scalar-valued case -- result is the following:
When the domain of $\mathcal{F}$ is extended from the nonconvex set ${\{ \mathbf{1}_u : u \in W^{1,1}(\Omega) \}}$ to the \emph{convex} relaxation
\begin{equation}
        \mathcal{C} := \left\{ v \in \BV(\Omega \times \Gamma, [0, 1]) : v(x, \min(\Gamma)) = 1, \, v(x, \max(\Gamma)) = 0 \right \},
        \label{eq:calibration_admissible_set}
\end{equation}
one can obtain minimizers of $\inf_{u \in W^{1,1}(\Omega)} \mathcal{F}(\mathbf{1}_u)$, which is a nonconvex problem, from minimizers of the convex problem $\inf_{v \in \mathcal{C}} \mathcal{F}(v)$ by a simple thresholding operation \cite[Thm. 3.1]{pock2010}.
Combined with the equivalence between~$F$ and~$\mathcal{F}$, this allows for a convex solution strategy to the original problem \eqref{eq:calibration_functional}.
This construction can be seen as a function-space equivalent of the graph construction by Ishikawa for discretized domains~\cite{ishikawa2003}.


\begin{figure}[t]
	\input{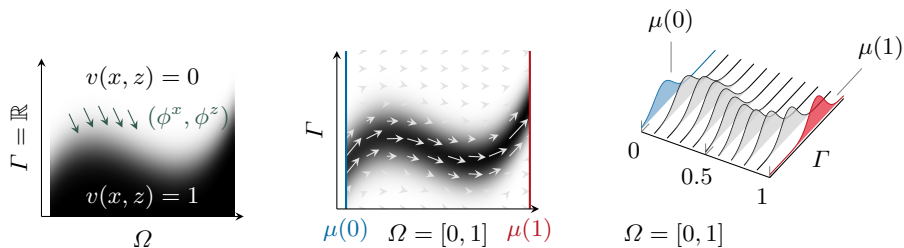}
	\caption{
		The theory of dynamical optimal transport exhibits conceptual
		similarities with state-of-the-art lifting strategies.
		The calibration method-based lifting (\textbf{left}) is defined on the subspace of
		functions $\BV(\Omega \times \R)$ that is defined in \eqref{eq:calibration_admissible_set}.
		A lifted functional is defined on the
		Cartesian product $\Omega \times \R$ via dual pairings $\langle\phi,Dv\rangle$
		with suitable vector fields $\phi = (\phi^x,\phi^z)$ as in \eqref{eq:calibration_dual_set}.
		In the Benamou-Brenier approach to dynamical optimal transport (\textbf{right}), the
		regularity of curves in the space of probability measures is measured
		by defining tangential vector fields on the Cartesian product
		$[0,1] \times \Gamma$.
		The derivative measures $Dv$ from the calibration-based model can be
		interpreted to correspond to the tangential vector fields in dynamical
		optimal transport (\textbf{center}).
		We will elaborate on this connection in Section~\ref{subsec:connections}.
	}
\label{fig:ot:calib-dyn}
\end{figure}

A generalization of the subgraph-based lifting strategy \eqref{eq:subgraph_fctn} to vectorial problems using the subgraphs of the separate channels of $u \in W^{1,1}(\Omega, \R^s)$ was investigated in \cite{strekalovskiy2014}, but remains limited to specific choices for the data term and the regularizer.

The majority of vectorial approaches resorts to alternative lifting approaches for vector-valued functions $u$.
Most commonly used are representations as pointwise Dirac measures \cite{lellmann2013b,loewenhauser2018,vogt2019b,vogt2019a} or, equivalently, Dirac distributions \cite{goldluecke2013} on $\Omega$ such that $\delta_{u(x)}$ concentrates a unit mass at $u(x) \in \Gamma$ for every $x \in \Omega$.
Convexity is then obtained through a relaxation to the set of pointwise probability measures.
Until now, however, this approach lacked a connection to the scalar-valued subgraph-lifting~\eqref{eq:calibration_lifting} and its theoretical justification through the calibration method~\cite{alberti2003} and could only be regarded as a heuristic.
Only recently has a fully continuous vectorial lifting strategy been proposed by \cite{mollenhoff2019} in which cartesian currents serve as lifted functions and thereby establish a link to the minimal surface problem that is at the heart of the calibration criterion.

Another avenue of generalizing the calibration method for \eqref{eq:calibration_functional} was taken in attempts to allow regularization of higher order.
Specifically, Laplacian regularization was investigated in \cite{loewenhauser2018,vogt2019a} for the application  case of image registration.
While \cite{loewenhauser2018} remained limited to an absolute Laplacian regularization through~${\| \Delta u(x) \|_1}$, the authors of \cite{vogt2019a} were able to overcome this drawback by allowing for a squared Laplacian penalization ${\| \Delta u(x) \|_2^2}$, also known as curvature regularization \cite{fischer2003}.
For scalar-valued signals $u$, higher-order regularization along the lines of the total generalized variation \cite{bredies2010} was discussed in \cite{ranftl2013,strecke2019}.

Lastly, we mention the sublabel-accurate discretization scheme for \eqref{eq:calibration_lifting} from \cite{mollenhoff2016,mollenhoff2017} that was later extended to vectorial problems in \cite{laude2016} and even manifold-valued problems in \cite{vogt2019b}.

\subsubsection{Dynamical Optimal Transport}


As the aim of this work is to point out the close parallels between functional lifting and the field of \emph{dynamical optimal transport}, we shortly recapitulate the basic formulations of optimal transport.

In its simplest variant, i.e., the Monge formulation~\cite{monge1784}, the problem reads as follows:
Given two probability measures~$\mu_0$ and~$\mu_1$ on $X \subset \R^d$, a transport map~${T: X \to X}$ is sought which minimizes the transport cost
\begin{equation}
	\int_{X} c(x, T(x)) \, \mu_0(\d x)
	\label{eq:monge_problem}
\end{equation}
under the constraint that $\mu_0(T^{-1}(A)) = \mu_1(A)$ holds for all measurable~${A \subset X}$.
Therein, $c: X \times X \to \R$ denotes a non-negative local cost function such as~${c(x, y) = \| x - y\|^p}$ for $p \geq 1$.

Since a solution~$T$ to~\eqref{eq:monge_problem} does not always exist, the more general Kantorovich formulation from~\cite{kantorovich1942}
\begin{equation}
	\inf \biggl\{ \int_{X \times X} c(x, y) \, \gamma(\d x, \d y) : \gamma \in \Pi(\mu_0, \mu_1) \biggr\},
	\label{eq:kantorovich_problem}
\end{equation}
\vspace{-2em}
\begin{multline}
	\Pi(\mu_0, \mu_1) := \{ \gamma \in \mathcal{P}(X \times X) :\\ \mu_0(A) = \gamma(A \times X), \mu_1(A) = \gamma(X \times A) \ \forall \text{ meas. } A \subset X \},
	\label{eq:ot_couplings}
\end{multline}
is often considered -- see~\cite{villani2009} for details.

\textit{Dynamical optimal transport} on the other hand refers to a \textit{time-dependent} formulation of~\eqref{eq:monge_problem} or, respectively,~\eqref{eq:kantorovich_problem} and was originally devised as a numerical scheme~\cite{benamou2000} for solving these problems in the case that~$\mu_0$ and~$\mu_1$ admit densities~$\rho_0$ and~$\rho_1$ with respect to the Lebesgue measure~$\mathcal{L}^d$.
The corresponding optimization problem for cost functions $c(x, y) = \| x - y\|^p/p$, $p \geq 1$, reads
\begin{multline}
	\inf_{\rho, m} \bigg\{ \int_{ X \times [0, 1] } \frac{\| m(x, t) \|^p}{p \rho(x, t)} \d (x, t) : \partial_t \rho + \ddiv_x m = 0, \rho(\cdot, 0) = \rho_0, \rho(\cdot, 1) = \rho_1 \bigg\},
	\label{eq:dynamical_ot_densities}
\end{multline}
where $\rho : X \times [0, 1] \to \R$ can be interpreted as an interpolation between~$\rho_0$ and~$\rho_1$ and where~${m : X \times [0, 1] \to \R^d}$ describes the momentum of the underlying mass transport.

Finally, a dynamical formulation of the transport problem for general probability measures $\mu_0, \mu_1$ which do not necessarily admit densities is introduced in~\cite{brenier2003} through duality: 
\begin{multline}
	\sup_{(\alpha, \beta) \in \mathcal{K}_q} \int_{X \times [0,1]} \alpha \d \mu + \int_{X \times [0,1]} \langle \beta, \d E \rangle \\
	\text{s.t.} \quad \partial_t \mu + \ddiv_x E = 0, \mu(0) = \mu_0, \ \mu(1) = \mu_1
	\label{eq:benamou_brenier}
\end{multline}
where $\mu : [0, 1] \to \mathcal{P}(X)$ and $E : [0, 1] \to \mathcal{M}(X, \R^d) := \mathcal{M}(X)^d$ assign probability measures and, respectively, vectorial Radon measures to each~${t \in [0, 1]}$.
The dual constraint set~$\mathcal{K}_q$ corresponding to the cost~${c(x, y) = \| x - y\|^p/p}$, $p > 1$, (analogously to~\eqref{eq:dynamical_ot_densities}) is given by
\begin{multline}
	\mathcal{K}_q := \biggl\{ (\alpha, \beta) \in C_b( X \times [0,1], \R \times \R^d) : \\
	\alpha(x, t) + \frac{ \| \beta(x, t) \|^q }{q} \leq 0 \ \forall (x, t) \in X \times [0,1] \biggr\}
	\label{eq:benamou_brenier_dual_set}
\end{multline}
with $q > 1$ as the dual exponent explained through $\frac{1}{p} + \frac{1}{q} = 1$.

Interestingly, the derivation of~\eqref{eq:benamou_brenier} in~\cite{brenier2003} starts out with pointwise Dirac measures $\mu(t) = \delta_{x_t}$, $x_t \in X$, and then extends the problem to arbitrary (pointwise) probability measures $\mu$ -- a congruency to the above-described approach in functional lifting.
While we discuss the connection between dynamical optimal transport and functional lifting in greater detail later on in this article, we already point out the structural similarities between the objective functionals~\eqref{eq:calibration_lifting} and~\eqref{eq:benamou_brenier} as well as the constraint sets~\eqref{eq:calibration_dual_set} and~\eqref{eq:benamou_brenier_dual_set}.

More extensive discussions of dynamical optimal transport are given in~\cite[Ch.~8]{ambrosio2008} as well as~\cite[Ch. 4-6]{santambrogio2015} from which we borrow the terminology of \textit{Benamou-Brenier functionals} for quantities of the form~\eqref{eq:benamou_brenier} and their generalizations.

One such generalization which will be relevant to our application is the extension to measure-valued mappings \textit{depending on multiple variables} instead of a sole dependence on time as in~\eqref{eq:benamou_brenier} and~\eqref{eq:benamou_brenier_dual_set}.
This scenario was first studied in~\cite[Sect.~3]{brenier2003} where a notion of harmonicity is introduced for functions $\mu : \Omega \to \mathcal{P}(X)$ on bounded $\Omega \subset \R^n$.
Recent advances on the existence of harmonic extensions for the corresponding Dirichlet problem were made in \cite{lavenant2019,lavenant2019phd}.
Note, however, that while an extension of the continuity equation from~\eqref{eq:benamou_brenier} to multidimensional domains~$\Omega$ is mathematically straightforward, its physical interpretation as a model for conserved quantities in spacetime (such as electric charge) is lost.

As standard references for optimal transport in general, we refer to the books~\cite{villani2009,santambrogio2015} and, for its computational aspects, to \cite{peyre2019}.

\subsection{Contribution}
\label{subsec:contribution}
In this article we propose a novel measure-valued lifting strategy for energies of the form~\eqref{eq:target_functional} which is inspired by dynamical optimal transport formulations.
We thereby lay the missing mathematical groundwork for the commonly used approach of using probability measures as lifted variables.

Our proposed framework features a modular structure based on the \textit{augmented reformulation} of~\eqref{eq:target_functional} as
\begin{equation}
        \tilde{F}(u, p) := \int_{\Omega} f(x, u(x), p(x)) \d x \quad \text{s.t.} \quad p = Lu.
        \label{eq:augmented_functional}
\end{equation}
This allows one to handle the lifting of $\tilde{F}(p, u)$ and that of the constraint $p = Lu$ in separate steps:
In the former case, we prove as a central theoretical contribution an integral representation for our generalized Benamou-Brenier functional~$\mathfrak{B}_f$ which leads to the equivalence $\mathfrak{B}_f((\delta_u, p \delta_u)) = \tilde{F}(u, p)$.

In the latter case, the constraint $p = Lu$ is translated into a corresponding condition for the lifted variables based on the particular choice of the differential operator $L$.
For each of the three operators $L = \nabla$, $L = \Delta$ and, as a novelty in functional lifting,~${L = \nabla^2}$, we provide these conditions in the form of weakly defined \textit{continuity equations} reminiscent of the one in~\eqref{eq:benamou_brenier}.

Furthermore, we establish detailed connections to the original scalar-valued lifting strategy~\eqref{eq:calibration_lifting} from \cite{pock2008,pock2010} -- see Figure~\ref{fig:ot:calib-dyn} for an illustration -- and to the notion of harmonic mappings in the Wasserstein space from~\cite{brenier2003,lavenant2019}.
As the central theoretical results of this article are based on one of the authors' thesis in \cite[Chapter~6]{vogt2020}, we also refer to that work for connections to several other lifting models such as the lifting strategy based on Cartesian currents from~\cite{mollenhoff2019}.

At last, we demonstrate the applicability of our approach as well as its compatibility with sublabel-accurate discretization schemes on a variety of numerical experiments on standard imaging problems.


\section{A Modular Framework for Measure-Valued Liftings}
\label{sec:framework}

\subsection{Preliminaries and Notation}
\label{subsec:preliminaries}

As we will be concerned with lifting problems of the form \eqref{eq:target_functional} to the space of measures on the product space $\Omega \times \Gamma$ for open and bounded $\Omega \subset \R^d$ and compact $\Gamma \subset \R^s$, we introduce $U := \Omega \times \Gamma$ and $n := d + s$ as notational shorthands. 
Accordingly, we will often write the integrand as $f(t, p)$ instead of $f(x, u, p)$ for $t = (x, u) \in U$. The convex conjugate $f^*(t, \xi)$ is always assumed to be taken with respect to the last variable $f^*(t, \xi) := \sup_{p \in \R^m} \langle \xi, p \rangle - f(t, p)$.

Bearing in mind the duality between the space $\mathcal{M}(U, \R^k)$ of vectorial Radon measures and the space $C_0(U, \R^k)$, i.e., the closure of the set of all compactly supported continuous functions from $U$ to $\R^k$, we define the \textit{graph-concentrated measure} $\delta_u$ corresponding to~${u : \Omega \to \Gamma}$ as
\begin{equation}
        \int_U \phi \d \delta_u := \int_{\Omega} \phi(x, u(x)) \d x
        \label{eq:delta_u_def}
\end{equation}
for all $\phi \in C_0(U) := C_0(U, \R)$.
Analogously, \textit{weakly measurable} measure-valued functions $\mu : \Omega \to \mathcal{P}(\Gamma)$ with $\mu_x := \mu(x)$ are defined as functions for which the mapping $x \mapsto \int_{\Gamma} \phi \d \mu_x$ is measurable on $\Omega$ for all $\phi \in C_0(\Gamma)$.
The set of all such functions is denoted by $L_w^{\infty}(\Omega, \mathcal{P}(\Gamma))$.

We begin our considerations with the trivial observation that if one optimizes over the functions $u\in\mathcal{U}$ and explicitly enforces the constraint $p=Lu$, the original problem and the augmented problem agree:
\begin{equation}
        \inf_{u \in \mathcal{U}} F(u) \quad = \quad \inf_{u \in \mathcal{U}} \tilde{F}(u, p) \ \ \text{s.t.} \ \ p = Lu
        \label{eq:augmentation_equivalence}
\end{equation}
for $F(u)$ from \eqref{eq:target_functional}, $\tilde{F}(u, p)$ from \eqref{eq:augmented_functional} and $\mathcal{U}$ as a suitable function space of sufficiently regular functions, e.g., $W^{1,1}(\Omega, \Gamma)$ in case of $L = \nabla$.
We will refer to~$\tilde{F}(u, p)$ as the \textit{augmented functional} and to $p = Lu$ as the \textit{augmentation constraint}.
In the following, these two components will be lifted separately which allows for a modular formulation.

\subsection{Benamou-Brenier Functional and Integral Representation}
\label{subsec:benamou_brenier_functional}
The augmented functional $\tilde{F}$ in \eqref{eq:augmented_functional} still inherits all undesired non-convexities of the original energy. In the following, we will instead represent the pair $(u,p)$, where $p$ should ultimately be forced to be equal to $Lu$,  by the \emph{measure} $(\delta_u,p\delta_u)$ over $V$ . The intention is that this will allow to construct a \emph{convex} functional~$\mathfrak{B}_f$, which we term \emph{generalized Benamou-Brenier functional}, so that
\beq
\mathfrak{B}_f ((E, \mu)) = \tilde{F}(u, p).
\eeq
We claim that such a $\mathfrak{B}_f$ can be constructed as follows:
\begin{defn}
        For $V \subset \R^n$, let $f: V \times \R^m \to [0, \infty]$ be lower semicontinuous in both variables and convex in the second variable.
        Then the \textup{generalized Benamou-Brenier functional} $\mathfrak{B}_f : \mathcal{M}(V, \R^{m+1}) \to [-\infty, \infty] =: \overline{\R}$ is defined as
        \begin{equation}
                \mathfrak{B}_f(\nu) := \sup_{\phi \in \mathcal{K}_f} \langle \nu, \phi \rangle,
                \label{eq:bb_functional}
        \end{equation}
        where $\langle \nu, \phi \rangle := \int_V \langle \phi, \d \nu \rangle$ is the
        dual pairing between $\nu$ and $\phi \in C_0(V, \R^{m+1})$ in the sense of Riesz-Markov and
        \begin{equation}
                \mathcal{K}_f := \{ (\phi^{\xi}, \phi^{\lambda}) \in C_0(V, \R^m \times \R) : f^*(t, \phi^{\xi}(t)) + \phi^{\lambda}(t) \leq 0 \ \forall t \in V \}
                \label{eq:bb_dual_set}
        \end{equation}
        is the set of \emph{dually admissible vector fields.}
        \label{def:generalized_bb}
\end{defn}

The lifted functional $\mathfrak{B}_f$ has much better properties than $F$ and $\tilde{F}$ with regards to existence, and it is even convex irrespective of convexity of $f$:

\begin{prop}
        Under the assumptions of Definition~\ref{def:generalized_bb}, the functional $\mathfrak{B}_f$ is nonnegative, convex and lower semicontinuous on $\mathcal{M}(V, \R^{m+1})$.
\end{prop}
\begin{proof}
        The nonnegativity of $f$ implies $0 \in \mathcal{K}_f$ and therefore $\mathfrak{B}_f(\nu) \geq \langle \nu, 0 \rangle = 0$.
        Since $\mathcal{M}(V, \R^{m+1})$ is the dual space of $C_0(V, \R^m \times \R)$, $\mathfrak{B}_f$ represents a pointwise supremum over a family of continuous linear functions and is thus both convex and lower semicontinuous. \qed
\end{proof}


To shed more light on the definition of $\mathfrak{B}_f$, we recall two definitions from finite-dimensional convex analysis:
the \textit{support function} $\sigma_S(x) := \sup_{y \in S} \langle x, y \rangle$ of a set $S \subset \R^n$, and the \textit{perspective function}
\begin{equation}
        h_g(\xi, \lambda) := \begin{cases}
        \lambda g(\xi / \lambda), & \lambda > 0,\\
        \lim_{\rho \searrow 0}  g(\xi / \rho), & \lambda = 0,\\
        +\infty, & \lambda < 0,
        \end{cases}
        \label{eq:perspective_fct}
\end{equation}
of a proper, convex and lower semicontinuous function $g : \R^{n-1} \to \overline{\R}$.
The two definitions are linked by the fact that the support function of the set
\beq
S_g := \{ (\xi, \lambda) \in \R^{n-1} \times \R : g^*(\xi) + \lambda \leq 0 \}
\eeq
is equal to the perspective function $h_g$~\cite[Corollary~13.5.1]{rockafellar1997}.

Note the similarity between the definitions of \eqref{eq:bb_dual_set} and $S_g$:
The former can roughly be viewed as a pointwise variant of the latter (for every point in $V$).
Likewise, the Benamou-Brenier functional~$\mathfrak{B}_f$ is defined analogously to the
support function $\sigma_S$.


In accordance with the above observations, we define the following:
\begin{defn}
        For $f: V \times \R^m \to [0, \infty]$ as in Definiton~\ref{def:generalized_bb}, its \textup{pointwise perspective function} $\tilde{h}_f: V \times \R^m \times \R \to [0, \infty]$ is given by
        \begin{equation}
                \tilde{h}_f(t, \xi, \lambda) := \begin{cases}
                \lambda f(t, \xi / \lambda), &\lambda > 0,\\
                \lim_{\rho \searrow 0} \rho f(t, \xi / \rho) , &\lambda = 0,\\
                +\infty, &\lambda < 0.
                \end{cases}
                \label{eq:pointwise_perspective}
        \end{equation}
        \label{def:pointwise_perspective}
\end{defn}

We are now ready to present our main result.
\begin{thm}
        Let $V \subset \R^n$ be locally compact and let $f: V \times \R^m \to [0, \infty]$ be as in Definition~\ref{def:generalized_bb} with the additional requirement that $t \mapsto f(t, 0)$ is locally bounded.
        Furthermore, let $\nu \in \mathcal{M}(V, \R^{m+1})$ and let the Radon-Nikodym density of $\nu$ with respect to its total variation $|\nu|$ be denoted by $(p^{\xi}, p^{\lambda}) \in L_{|\nu|}^1 (V, \R^m \times \R)$.
        Then, one has the integral representation
        \begin{equation}
                \mathfrak{B}_f(\nu) = \int_V \tilde{h}_f \left( t, p^{\xi}(t), p^{\lambda}(t) \right) \, |\nu|(\mathrm{d} t).
                \label{eq:bb_integral_representation}
        \end{equation}
        \label{thm:bb_integral_representation}
\end{thm}

Local compactness of $V \subset \R^n$ in Theorem~\ref{thm:bb_integral_representation} is defined through the criterion, that every $t \in V$ has a compact neighborhood in $V$ -- in our application case, namely $V = U = \Omega \times \Gamma$, this condition will always be met thanks to the compactness of $\Gamma$~\cite[Theorem~18.6]{willard1970}.
Likewise, the boundedness assumption on~$f$ follows from the typical data-term/regularizer structure $f(t, p) = \rho(t) + \eta(p)$ in which the regularizer commonly satisfies $\eta(0) = 0$ while the data-term~$\rho$ is bounded on $U$.

The motivation behind this assumption is the fact that it implies lower semicontinuity and thereby measurability of $\tilde{h}_f$ as it is required in order for the integral~\eqref{eq:bb_integral_representation} to be well-defined.
Precisely, one has the following connection dating back to~\cite[Theorem~3.1]{dalmaso1979}.
\begin{lem}
        The perspective function $\tilde{h}_f: V \times \R^m \times \R \to [0, \infty]$ of a function~${f: V \times \R^m \to [0, \infty]}$ as in Theorem~\ref{thm:bb_integral_representation} is lower semicontinuous.
        \label{lem:perspective_lsc}
\end{lem}
\begin{proof}
        Lower semicontinuity of $\tilde{h}_f$ holds trivially for all $(t, \xi, \lambda)$ with $\lambda < 0$ and, by the lower semicontinuity of $f$, also for $\lambda > 0$.
        It remains to show that~$\tilde{h}_f$ is lower semicontinuous at $(t^*, \xi^*, 0)$, i.e., that for every $(t^*, \xi^*) \in V \times \R^m$ and every $M < \tilde{h}_f (t^*, \xi^*, 0)$ there exists a neighborhood $\mathcal{N}$ of $(t^*, \xi^*, 0)$ such that~${M \leq \tilde{h}_f (t, \xi, \lambda)}$ for all $(t, \xi, \lambda) \in \mathcal{N}$.
        
        Let $\epsilon > 0$ be fixed.
        By the definition~$\tilde{h}_f (t^*, \xi^*, 0) = \lim_{\rho \searrow 0} \rho f(t^*, \xi^* / \rho)$, there exists $\delta \in ]0, \epsilon[$ with
        \begin{equation}
                \tilde{h}_f (t^*, \xi^*, 0) - \epsilon < \tilde{h}_f (t^*, \xi^*, \tau) \quad \forall \tau \in [0, \delta].
                \label{eq:h_lsc_proof_1}
        \end{equation}
        By the lower semicontinuity of $\tilde{h}_f$ at $(t^*, \xi^*, \delta)$, there exists a neighborhood $\mathcal{N}'$ of $(t^*, \xi^*)$ such that
        \begin{equation}
                \tilde{h}_f (t^*, \xi^*, \delta) - \epsilon < \tilde{h}_f (t, \xi, \delta) \quad \forall (t, \xi) \in \mathcal{N}'.
                \label{eq:h_lsc_proof_2}
        \end{equation}
        By the convexity of $f$ in the last argument as well as its nonnegativity, it holds for all $\tau \in ]0, \delta[$ that
        \begin{align}
                \tilde{h}_f(t, \xi, \delta) &= \delta f(t, \xi / \delta) = \delta f(t, (\xi / \tau) (\tau / \delta)) \label{eq:h_lsc_proof_3} \\
                &\leq \tau f(t, \xi / \tau) + (\delta - \tau) f(t, 0) \label{eq:h_lsc_proof_4} \\
                &\leq \tilde{h}_f (t, \xi, \tau) + \epsilon f(t, 0). \label{eq:h_lsc_proof_5}
        \end{align}
        Together with \eqref{eq:h_lsc_proof_1} and \eqref{eq:h_lsc_proof_2}, one obtains
        \begin{equation}
                \tilde{h}_f (t^*, \xi^*, 0) - \epsilon (2 + f(t, 0)) \leq \tilde{h}_f (t, \xi, \tau).
                \label{eq:h_lsc_proof_6}
        \end{equation}
        As $t \mapsto f(t, 0)$ is assumed to be locally bounded, one can introduce an upper bound $C \geq 2 + f(t, 0) \geq 0$ on a suitable neighborhood $\mathcal{N}''$ of $t^*$, so that any lower bound
        \begin{equation}
                M := \tilde{h}_f (t^*, \xi^*, 0) - \epsilon C < \tilde{h}_f (t^*, \xi^*, 0)
        \end{equation}
        can be achieved for $\tilde{h}_f (t, \xi, \tau)$ on $\mathcal{N} := \{ (t, \xi, \tau) | (t, \xi) \in \mathcal{N}', t \in \mathcal{N}'', \tau \in [0, \delta[ \}$ by the arbitrariness of $\epsilon$.
        Since $\tilde{h}_f (t, \xi, \tau) = +\infty$ holds for all points with $\tau < 0$, one can extend $\mathcal{N}$ by the corresponding orthant, which concludes the proof. \qed
\end{proof}

\subsection{Proof of Theorem~\ref{thm:bb_integral_representation}}
\label{subsec:bb_integral_representation_proof}

\begin{figure}[t]
	\centering
	\begin{tikzpicture}
	
	\draw[->] (0, 0) -- (4, 0) node[above] {$X$};
	\draw[->] (0, 0) -- (0, 3.5) node[right] {$Y$};
	
	\filldraw[draw=none, gray, opacity=0.1] (0, 2.6) -- (4, 2.6) -- (4, 3.25) -- (0, 3.25);
	\draw[gray, dashed] (0, 2.6) -- (4, 2.6);
	\draw[gray, dashed] (4, 3.25) -- (0, 3.25);
	
	\filldraw[draw=none, gray, opacity=0.1] (1, 0) -- (1, 3.5) -- (2.33, 3.5) -- (2.33, 0);
	\draw[gray, dashed] (2.33, 3.5) -- (2.33, 0);
	\draw[gray] (1, 0) -- (1, 3.5);
		
	\filldraw[draw=none, blue, opacity=0.1] (0, 0.9) -- (0.5, 1.1) -- (1, 1) -- (1, 0.5) -- (3, 0.5) -- (3, 1.5) -- (4, 1.6) -- (4, 2.4) -- (3, 2.4) -- (1, 3) -- (1, 2) -- (0, 2);
	\draw[blue, thick] (0, 0.9) -- (0.5, 1.1) -- (1, 1) -- (1, 0.5) -- (3, 0.5) -- (3, 1.5) -- (4, 1.6);
	\draw[blue, thick] (4, 2.4) -- (3, 2.4) -- (1, 3) -- (1, 2) -- (0, 2);
	
	\node at (1, 0) {\small [};
	\node at (2.33, 0) {\small )};
	\draw[thick] (1, 0) -- (2.33, 0) node[midway, below] {$Q^{-1}(V)$};
	
	\node[rotate=90] at (0, 2.6) {\small (};
	\node[rotate=90] at (0, 3.25) {\small )};
	\draw[thick] (0, 2.6) -- (0, 3.25) node[midway, left] {$V$};
	
	\draw[red, thick] (3.2, 1.52) -- (3.2, 2.4) node[midway, right] {\small $Q(x)$};
	\node at (3.25, 0) {\small $\vert$};
	\node[below=3pt] at (3.25, 0) {$x$};
	
	\node[above=12pt] at (2, 3.5) {$Q$ \textbf{not} inner semicontinuous};
	
	
	\draw[->] (6, 0) -- (10, 0) node[above] {$X$};
	\draw[->] (6, 0) -- (6, 3.5) node[right] {$Y$};
	
	\filldraw[draw=none, gray, opacity=0.1] (6, 2.6) -- (10, 2.6) -- (10, 3.25) -- (6, 3.25);
	\draw[gray, dashed] (6, 2.6) -- (10, 2.6);
	\draw[gray, dashed] (10, 3.25) -- (6, 3.25);
	
	\filldraw[draw=none, gray, opacity=0.1] (7, 0) -- (7, 3.5) -- (8.33, 3.5) -- (8.33, 0);
	\draw[gray, dashed] (8.33, 3.5) -- (8.33, 0);
	\draw[gray, dashed] (7, 0) -- (7, 3.5);
	
	\filldraw[draw=none, blue, opacity=0.1] (6, 0.9) -- (6.5, 1.1) -- (7, 1) -- (7, 0.5) -- (9, 0.5) -- (9, 1.5) -- (10, 1.6) -- (10, 2.4) -- (9, 2.4) -- (7, 3) -- (7, 2) -- (6, 2);
	\draw[blue, thick] (6, 0.9) -- (6.5, 1.1) -- (7, 1);
	\draw[blue, thick, dashed] (7, 1) -- (7, 0.5);
	\draw[blue, thick] (7, 0.5) -- (9, 0.5);
	\draw[blue, thick, dashed] (9, 0.5) -- (9, 1.5);
	\draw[blue, thick] (9, 1.5) -- (10, 1.6);
	\draw[blue, thick] (10, 2.4) -- (9, 2.4) -- (7, 3);
	\draw[blue, thick, dashed] (7, 3) -- (7, 2);
	\draw[blue, thick] (7, 2) -- (6, 2);
	
	\node at (7, 0) {\small (};
	\node at (8.33, 0) {\small )};
	\draw[thick] (7, 0) -- (8.33, 0) node[midway, below] {$Q^{-1}(V)$};
	
	\node[rotate=90] at (6, 2.6) {\small (};
	\node[rotate=90] at (6, 3.25) {\small )};
	\draw[thick] (6, 2.6) -- (6, 3.25) node[midway, left] {$V$};
	
	\draw[olive, thick] plot [smooth, tension=0.5] coordinates {(6, 1.45) (7, 1.35) (8, 1.75) (9, 2) (10, 1.9)};
	\node[olive] at (8.25, 1.375) {\small $\phi \in C(Q)$};
	
	\node[above=12pt] at (8, 3.5) {$Q$ inner semicontinuous};
	
\end{tikzpicture}
	\caption{Visualization of a set-valued mapping $Q: X \rightrightarrows Y$ and the concept of inner semicontinuity.
		The preimage of every open set $V \subset Y$ under $Q$, i.e., the set of points~${x \in X}$ for which $Q(x)$ intersects $V$, needs to be an open set in order for $Q$ to be inner semicontinuous.
		\textbf{Left:} As the graph of $Q$ (blue) contains vertical boundaries, the preimage of the displayed set $V$ contains a boundary point -- therefore, $Q$ is not inner semicontinuous.
		\textbf{Right:} Once these vertical boundaries are excluded from $Q$, the displayed mapping becomes inner semicontinuous.
		Furthermore, a continuous selection~$\phi$ of $Q$ is shown, i.e., a continuous function $\phi: X \to Y$ with $\phi(x) \in Q(x)$ for all $x \in X$.
		Under suitable assumptions on $X$ and $Y$, the existence of such a function is guaranteed for inner semicontinuous $Q$ by Michael's theorem (see Theorem~\ref{thm:michael}).
	}
	\label{fig:inner_semicontinuity}
\end{figure}
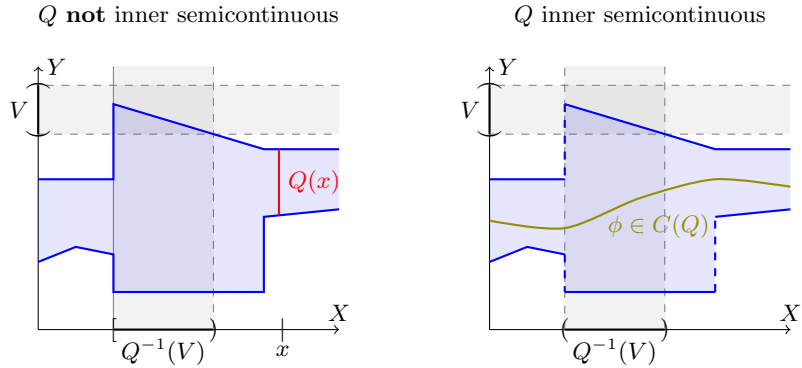

This section is concerned with a proof of Theorem~\ref{thm:bb_integral_representation}, i.e., the integral representation~\eqref{eq:bb_integral_representation} of the Benamou-Brenier functional, based on the recent results of~\cite{perkkio2018}.
Readers who are primarily interested in the application aspects of this article may safely skip ahead to Section~\ref{subsec:lifting_augmented} where we demonstrate the use of~$\mathfrak{B}_f$ for functional lifting.

We proceed our proof by showcasing that Theorem~\ref{thm:bb_integral_representation} is a special case of Theorem~1 from~\cite{perkkio2018}.
Before we can state that latter result, we first introduce the relevant prerequisites:
We denote a \textit{set-valued mapping} $Q$ from one general set~$X$ to another~$Y$, i.e.,~${Q(x) \subset Y}$ for all $x \in X$, by $Q: X \rightrightarrows Y$.
Correspondingly, the \textit{preimage} of $V \subseteq Y$ under such a mapping~$Q$ is defined as
\begin{equation}
        Q^{-1}(V) := \{ x \in X : Q(x) \cap V \neq \emptyset \}.
\end{equation}

Furthermore, we require a notion of continuity for set-valued mappings on topological spaces~$X$ and~$Y$. We call $Q: X \rightrightarrows Y$ \textit{inner semicontinuous} if~$Q^{-1}(V)$ is open for each open set $V \subset Y$.
Note that inner semicontinuity for single-valued functions does not equal lower or upper semicontinuity, but rather ``plain'' continuity of functions between topological spaces.
If $Y$ is a normed vector space, the set of \textit{continuous selections} is defined as
\begin{equation}
        C(Q) := \{ \phi \in C_0(X, Y) : \phi(x) \in Q(x) \ \forall x \in X \}.
\end{equation}
See Figure~\ref{fig:inner_semicontinuity} for a visualization of these concepts.

Lastly, a \textit{second-countable and locally compact Hausdorff space} -- the type of space that~\cite{perkkio2018} is concerned with -- is a topological space in which points can be separated by disjoint neighborhoods (\textit{Hausdorff property}), in which every point has a compact neighborhood (\textit{local compactness property}) and in which a countable collection $\mathcal{U}$ of open sets exists such that every open set can be written as the union of sets from $\mathcal{U}$ (\textit{second countability property}).
For our application scenario, i.e., $\Omega \times \Gamma \subset \R^d \times \R^s$ as in Section~\ref{subsec:preliminaries}, the question of local compactness has already been resolved in the previous Section~\ref{subsec:benamou_brenier_functional}.
As for the Hausdorff and second-countability properties, both are always satisfied for separable metric spaces \cite[Theorem~16.11]{willard1970} or, more specifically, for all subsets of Euclidean spaces.

We are now ready to present the central result from~\cite{perkkio2018}.

\begin{thm}[{\cite[Theorem~1]{perkkio2018}}]
        Let $X$ be a second-countable and locally compact Hausdorff space, let $Q: X \rightrightarrows \R^m$ be a set-valued mapping for which the set of continuous selections $C(Q)$ is nonempty.
        Furthermore, let the set $Q(x)$ be closed and convex for all $x \in X$.
        Then, the following two statements are equivalent:
        \begin{enumerate}
                \item The support function of $C(Q)$ on $\mathcal{M}(X, \R^m)$ satisfies
                \begin{equation}
                        \sigma_{C(Q)}(\nu) := \sup_{\phi \in C(Q)} \langle \nu, \phi \rangle = \int_X \sigma_{Q(x)} ((\mathrm{d} \nu) / (\mathrm{d} |\nu|) (x)) \, |\nu|(\mathrm{d} x)
                        \label{eq:perkkio_integral_representation}
                \end{equation}
                for all $\nu \in \mathcal{M}(X, \R^m)$.
                \item $Q$ is inner semicontinuous.
        \end{enumerate}
        \label{thm:perkkio_integral_representation}
\end{thm}

On a historical note, an earlier form of this result has already been presented in~\cite[Theorem~6]{rockafellar1971}.
That variant, however, required $X$ to be compact (instead of locally compact) and showed only the sufficient condition, i.e., that condition~\textit{2.} implies condition \textit{1.} from Theorem~\ref{thm:perkkio_integral_representation}.

For our proof of Theorem~\ref{thm:bb_integral_representation}, we will work on $X = U = \Omega \times \Gamma$ with a set-valued mapping~${Q : U \rightrightarrows \R^{m+1}}$ defined by
\begin{equation}
        Q(t) := \{ (\xi, \lambda) \in \R^m \times \R : \lambda + f^*(t, \xi) \leq 0 \}.
        \label{eq:Q_f}
\end{equation}
Since the conjugate $f^*(t, \cdot)$ is a proper, lower semicontinuous and convex function for every $t \in U$, it is easy to see that the corresponding set~$Q(t)$ is non-empty, closed and convex.
Hence, it remains to show that $C(Q)$ is non-empty and that~$Q$ is inner semicontinuous.

As it turns out in our setting, however, inner semicontinuity of~$Q$ is already sufficient for the existence of a continuous selection by Michael's theorem:

\begin{thm}[{\cite[Theorem~3.2'']{michael1956}}] 
        Assume that $X$ is a separable metric space and that $Q: X \rightrightarrows \R^m$ is inner semicontinuous with non-empty, closed and convex values for all $x \in X$.
        Then, $C(Q) \neq \emptyset$.
        \label{thm:michael}
\end{thm}

Finally, inner semicontinuity of $Q$ is equal to lower semicontinuity of the integrand~$f$.
The proof we present is based on the arguments given for~\cite[Theorem~8]{bouchitte1988}.

\begin{prop}
        Let $V \subset \R^n$ be nonempty, let $f: V \times \R^m \to \R$ be convex in the second variable with locally bounded $t \mapsto f(t, 0)$ and $Q: U \rightrightarrows \R^{m+1}$ as in~\eqref{eq:Q_f}.
        Then, lower semicontinuity of~$f$ is equivalent to inner semicontinuity of~$Q$.
        \label{prop:Q_isc_f_lsc}
\end{prop}
\begin{proof}
        We begin by showing that lower semicontinuity of~$f$ implies inner semicontinuity of $Q$.
        Hence, we assume that $f$ is lower semicontinuous and, for the sake of contradiction, that $Q$ is not inner semicontinuous.
        Consequently, there exists an open set $W \subset \R^{m+1}$ such that $Q^{-1}(W)$ is not open, i.e., $W$ contains a boundary point $t^*$.
        Without loss of generality, we can assume that $W = B_{\epsilon}(\xi, \lambda)$ for some $(\xi, \lambda) \in Q(t^*)$ and $\epsilon > 0$.
        Since $t^*$ is a boundary point of $B_{\epsilon}(\xi, \lambda)$, there exists a sequence $t^k \to t^*$ in $V$ with $Q(t^k) \cap B_{\epsilon}(\xi, \lambda) = \emptyset$.
        After subtracting the linear and lower semicontinuous function $\zeta \mapsto \langle \xi, \zeta \rangle + \lambda$ from $f$, we can assume that $(\xi, \lambda) = 0 \in Q(t^*)$.
        By the above definition of $(t^k)_k$, we have that the convex sets $Q(t^k)$ and $B_{\epsilon}(0)$ are disjoint for every $k \in \N$, so that the hyperplane separation theorem can be applied to these two sets, i.e., there exist $(\xi^k, \lambda^k)$ with
        \begin{equation}
                \sup_{(\xi, \lambda) \in Q(t^k)} \langle (\xi^k, \lambda^k), (\xi, \lambda) \rangle \leq -1 \leq \inf_{(\xi, \lambda) \in B_{\epsilon}(0)} \langle (\xi^k, \lambda^k), (\xi, \lambda) \rangle.
                \label{eq:hyperplane_separation}
        \end{equation}
        Recalling the fact that the support function of a set of the form of $Q(t^k)$ is a perspective function (see p.~\pageref{eq:perspective_fct}), we see that the left hand side of~\eqref{eq:hyperplane_separation} is equal to $\tilde{h}_f(t^k, \xi^k, \lambda^k)$.
        For the right hand side, we compute
        \begin{equation}
                \inf_{(\xi, \lambda) \in B_{\epsilon}(0)} \langle (\xi^k, \lambda^k), (\xi, \lambda) \rangle = - \epsilon \sup_{(\xi, \lambda) \in B_1(0)} \langle (\xi^k, \lambda^k), (\xi, \lambda) \rangle = - \epsilon \| (\xi^k, \lambda^k) \|.
        \end{equation}
        Therefore, $\| (\xi^k, \lambda^k) \| \leq \epsilon^{-1}$ holds for every $k$, so that $(\xi^k, \lambda^k) \to  (\xi^*, \lambda^*)$ for some subsequence.
        Furthermore, we have by Lemma~\ref{lem:perspective_lsc} that $\tilde{h}_f$ is lower semicontinuous under the above assumptions on $f$.
        This yields
        \begin{equation}
                \tilde{h}_f(t^*, \xi^*, \lambda^*) \leq \liminf_{k \to \infty} h(t^k, \xi^k, \lambda^k) \leq -1.
                \label{eq:Q_isc_eq1}
        \end{equation}
        However, as stated above, $0 \in Q(t^*)$, so that $\tilde{h}_f(t^*, \xi^*, \lambda^*) \geq 0$ holds by the support function representation of $\tilde{h}_f$.
        Clearly, this is a contradiction to~\eqref{eq:Q_isc_eq1}.
        
        Next, we show that inner semicontinuity of~$Q$ implies lower semicontinuity of~$f$.
        To this end, fix $t^* \in V$, $\zeta \in \R^m$ and $M < f(t^*, \zeta^*)$.
        Our goal is to show the existence of $\delta > 0$ with $f(t, \zeta) \geq M$ for all $t \in B_{\delta}(t^*)$ and $\zeta \in B_{\delta}(\zeta)$.
        Choose $\epsilon > 0$ such that $M + 2 \epsilon < f(t^*, \zeta^*)$.
        Then, the point $(\zeta^*, M + 2 \epsilon) \in \R^m \times \R$ has a positive distance from the (epi-)graph of the function $f(t^*, \cdot)$ and, since $f(t^*, \cdot)$ is convex, there exists an affine linear function that separates $(\zeta^*, M + 2 \epsilon)$ from that (epi-)graph.
        As $Q(t^*)$ consists of the slope/intercept-parameter pairs of all affine linear functions underneath $f(t^*, \cdot)$ it contains, in particular, a parameter pair $(\xi^*, \lambda^*)$ with
        \begin{equation}
                \langle \xi^*, \zeta^* \rangle + \lambda^* \geq M + 2 \epsilon .
                \label{eq:Q_isc_eq2}
        \end{equation}
        Consider the following open neighborhood of $(\xi^*, \lambda^*)$:
        \begin{equation}
                W := \{ (\xi, \lambda) \in \R^{m+1} : \| \xi - \xi^* \| + | \lambda - \lambda^* | < \epsilon / \max\{ 1, \| \zeta^* \| \} \}.
                \label{eq:Q_isc_eq3}
        \end{equation}
        By the inner semicontinuity of $Q$, $Q^{-1}(W)$ is open and, as $(\xi^*, \lambda^*) \in Q(t^*)$, it contains an open neighborhood $B_{\delta'}(t^*)$ of $t^*$ for some $\delta' > 0$.
        Furthermore, one has~${Q(t) \cap W \neq \emptyset}$ for all $t \in B_{\delta'}(t^*)$ by the definition of the preimage, which means we can assign such a point $(\xi(t), \lambda(t)) \in Q(t) \cap W$ to each $t \in B_{\delta'}(t^*)$.
        This implies
        \begin{equation}
                f(t, \zeta) \geq \langle \xi(t), \zeta \rangle + \lambda(t) \quad \forall \zeta \in \R^m.
        \end{equation}
        We continue by computing
        \begin{align}
                f(t, \zeta) &\geq \langle \xi(t), \zeta \rangle + \lambda(t)\\
                &= \langle \xi(t), \zeta - \zeta^* \rangle + \langle \xi(t) - \xi^*, \zeta^* \rangle + (\lambda(t) - \lambda^*) + \langle \xi^*, \zeta^* \rangle + \lambda^*\\
                &\geq \langle \xi(t), \zeta - \zeta^* \rangle + \langle \xi(t) - \xi^*, \zeta^* \rangle + (\lambda(t) - \lambda^*) + M + 2 \epsilon\\
                &\geq \langle \xi(t), \zeta - \zeta^* \rangle - \| \xi(t) - \xi^* \| \| \zeta^* \| - | \lambda(t) - \lambda^* | + M + 2 \epsilon
        \end{align}
        by \eqref{eq:Q_isc_eq2} and the Cauchy-Schwarz inequality, respectively.
        Due to $(\xi(t), \lambda(t)) \in W$, we further obtain
        \begin{align}
                f(t, \zeta) &\geq M + \epsilon + \langle \xi(t), \zeta - \zeta^* \rangle\\
                &\geq M + \epsilon - \| \xi(t) \| \|\zeta - \zeta^* \|
        \end{align}
        and, as $\| \xi(t) \|$ can be bounded from below again by~\eqref{eq:Q_isc_eq3},
        \begin{equation}
                f(t, \zeta) \geq M + \epsilon - (\| \zeta^* \| + \epsilon) \| \zeta - \zeta^* \|.
        \end{equation}
        At last, we define $\delta := \min\{ \delta', \epsilon / (\| \xi^* \| + \epsilon) \}$ -- implying $(\| \xi^* \| + \epsilon) \| \zeta - \zeta^* \| \leq \epsilon$ for all $\zeta \in B_{\delta}(\zeta^*)$ -- and conclude $f(t, \zeta) \geq M$ for all $t \in B_{\delta}(t^*)$ and $\zeta \in B_{\delta}(\zeta^*)$, which proves the lower semicontinuity of $f$ at $(t^*, \zeta^*)$. \qed
\end{proof}

As discussed above, Proposition~\ref{prop:Q_isc_f_lsc} implies the validity of Theorem~\ref{thm:perkkio_integral_representation} for $Q$ as defined in \eqref{eq:Q_f}.
The facts that the set $C(Q)$ of continuous selections of $Q$ equals the constraint set $\mathcal{K}_f$ from Definition~\ref{def:generalized_bb} and that the support function~$\sigma_{Q(x)}$ from~\eqref{eq:perkkio_integral_representation} equals the (pointwise) perspective function~$\tilde{h}_f$ prove Theorem~\ref{thm:bb_integral_representation} as a special case of Theorem~\ref{thm:perkkio_integral_representation}. 

\subsection{Lifting the Augmented Functional}
\label{subsec:lifting_augmented}

As our goal is to use $\mathfrak{B}_f$ to represent the augmented functional $\tilde{F}(u, p)$ equivalently through variables that are measures, we will now establish their exact connection through Theorem~\ref{thm:bb_integral_representation}.
Analogously to the manner, in which $\tilde{F}$ depends on the two variables $u$ and $p$, we consider in this section split measures~${\nu = (E, \mu) \in \mathcal{M}(U, \R^m \times \R)}$ as arguments of $\mathfrak{B}_f$.
We begin by presenting a refinement of Theorem~\ref{thm:bb_integral_representation} to this case.
\begin{prop}
        Suppose that the assumptions of Theorem~\ref{thm:bb_integral_representation} hold.
        Furthermore, let $\nu = (E, \mu) \in \mathcal{M}(V, \R^m \times \R)$ and let the \textup{Lebesgue decomposition} of $E$ with respect to $|\mu|$ be given by $E = v |\mu| + v^{\bot} |\mu|^{\bot}$ where~$|\mu|^{\bot} \geq 0$ is the part of $E$ that is singular with respect to $|\mu|$ and where $v \in L^1_{|\mu|}(V, \R^m)$, $v^{\bot} \in L_{|\mu|^{\bot}}(V, \R^m)$ are the respective densities.
        Then, it holds that
        \begin{equation}
                \mathfrak{B}_f((E, \mu)) = \begin{cases}
                \int_V f(t, v(t)) \, |\mu|(\mathrm{d} t) + \int_V f^{\infty}(t, v^{\bot}(t)) \, |\mu|^{\bot}(\mathrm{d} t), & \mu \geq 0\\
                +\infty, & \text{otherwise},
                \end{cases}
                \label{eq:bb_E_mu}
        \end{equation}
        where $f^{\infty}(t, \xi) := \lim_{\rho \searrow 0} \rho f(t, \xi / \rho)$.
\end{prop}
\begin{proof}
        By Theorem~\ref{thm:bb_integral_representation}, we have
        \begin{equation}
                \mathfrak{B}_f((E, \mu)) = \int_V \tilde{h}_f \left( t, p^{\xi}(t), p^{\lambda}(t) \right) \, |\nu|(\mathrm{d} t)
        \end{equation}
        for the Radon-Nikodym densities $p^{\xi} = \mathrm{d} E / \mathrm{d} |\nu|$ and $p^{\lambda} = \mathrm{d} \mu / \mathrm{d} |\nu|$.
        From the definition of $\tilde{h}_f$ in \eqref{eq:pointwise_perspective}, we deduce $\mathfrak{B}_f((E, \mu)) = +\infty$ in case that a Borel set~${A \subset V}$ with $|\nu|(A) > 0$ and $\mu(A) < 0$ exists.
        
        Otherwise, one can use the fact that the pointwise perspective function $\tilde{h}_f$ is one-homogeneous in the last two arguments:
        Since $|\mu| = \mu$ and since the density~${\mathrm{d} |\mu| / \mathrm{d} |\nu|}$ trivially exists on $S := \supp |\mu|$, one has
        \begin{align}
                \int_{S} \tilde{h}_f \left( t, \frac{\mathrm{d} E}{\mathrm{d} |\nu|} (t), \frac{\mathrm{d} \mu}{\mathrm{d} |\nu|} (t) \right) \, |\nu|(\mathrm{d} t) &=
                \int_{S} \tilde{h}_f \left( t, \frac{\mathrm{d} E}{\mathrm{d} |\mu|} (t), 1 \right) \frac{\mathrm{d} |\mu|}{\mathrm{d} |\nu|} (t) \, |\nu|(\mathrm{d} t) \\
                 &= \int_{S} f(t, v(t)) \, |\mu|(\mathrm{d} t) \label{eq:int_supp_mu}
        \end{align}
        by the chain rule for Radon-Nikodym densities.
        On $S^{\bot} := \supp |\mu|^{\bot}$, on the other hand, one has a density $\mathrm{d} |\mu|^{\bot} / \mathrm{d} |\nu|$ with
        \begin{align}
                \int_{S^{\bot}} \tilde{h}_f \left( t, \frac{\mathrm{d} E}{\mathrm{d} |\nu|} (t), \frac{\mathrm{d} \mu}{\mathrm{d} |\nu|} (t) \right) \, |\nu|(\mathrm{d} t) &= 
                \int_{S^{\bot}} \tilde{h}_f \left( t, \frac{\mathrm{d} E}{\mathrm{d} |\mu|^{\bot}} (t), 0 \right) \frac{\mathrm{d} |\mu|^{\bot}}{\mathrm{d} |\nu|} (t) \, |\nu|(\mathrm{d} t) \\
                 &= \int_{S^{\bot}} f^{\infty} (t, v^{\bot} (t)) \, |\mu|^{\bot}(\mathrm{d} t). \label{eq:int_supp_mu_bot}
        \end{align}
        Assuming without loss of generality that $\supp E = S \cup S^{\bot}$ allows one to conclude~\eqref{eq:bb_E_mu} from \eqref{eq:int_supp_mu} and \eqref{eq:int_supp_mu_bot}.\qed
\end{proof}

Our particular case of interest will be $E \ll \mu$, i.e., $E$ being absolutely continuous with respect to $\mu$ in the sense of $E(A) = 0$ for all Borel sets $A$ with~$\mu(A) = 0$. In this case the roles of $E$ in $\mathfrak{B}_f$ and $p$ in $\tilde{F}$ can then be linked in a meaningful sense.
As a first step, we deduce from \eqref{eq:bb_E_mu} that
\begin{equation}
        \mathfrak{B}_f((E, \mu)) = \int_V f(t, v(t)) \, \mu(\mathrm{d} t)
        \label{eq:bb_E_mu_ac}
\end{equation}
holds for $\mu \geq 0$ and $E \ll \mu$.
As a second step, we conclude the exact connection between $\mathfrak{B}_f((E, \mu))$ and $\tilde{F}(u, p)$. 
\begin{cor}
        For $U = \Omega \times \Gamma$ equipped with the usual assumptions and for measurable $u: \Omega \to \Gamma$, $p: \Omega \to \R^m$ with $\tilde{F}(u, p) < \infty$, the vectorial measure~$(E, \mu) \in \mathcal{M}(U, \R^m \times \R)$ with $\mu = \delta_u$ in the sense of~\eqref{eq:delta_u_def} and $E = p \delta_u$ in the sense of
        \begin{equation}
                \int_U \langle \phi, \mathrm{d} E \rangle := \int_{\Omega} \langle \phi(x, u(x)), p(x) \rangle \d x
        \end{equation}
        for all $\phi \in C_0(U, \R^m)$ satisfies
        \begin{equation}
                \mathfrak{B}_f((E, \mu)) = \int_{\Omega} f(x, u(x), p(x)) \d x = \tilde{F}(u, p).
                \label{eq:bb_lifting}
        \end{equation}
        \label{cor:bb_lifting}
\end{cor}

Equations~\eqref{eq:bb_E_mu_ac} and~\eqref{eq:bb_lifting} emphasize that absolute continuity in the above sense is central to the generalized Benamou-Brenier functional $\mathfrak{B}_f$ -- a fact that is reminiscent of the ``original'' Benamou-Brenier functional from~\cite[Ch.~5]{santambrogio2015}.
What is more, the generalized Benamou-Brenier functional even enforces absolute continuity under certain conditions as the following proposition shows.
\begin{prop}
        Let $V \subset \R^n$ be locally compact and let $f: V \times \R^m \to \overline{\R}$ be convex in the second argument with a continuous conjugate $f^*$ that satisfies
        \begin{equation}
                \sup_{t \in V, \| \xi \| < r} |f^*(t, \xi)| < \infty \quad \forall r > 0.
        \end{equation}
        Then, finiteness of $\mathfrak{B}_f((E, \mu))$ for $(E, \mu) \in \mathcal{M}(V, \R^m \times \R)$ implies $E \ll \mu$.
        \label{prop:enforced_ac}
\end{prop}

In order to prove the above Proposition, we first show the following Lemma.
\begin{lem}
        Under the assumptions of Proposition~\ref{prop:enforced_ac}, one has
        \begin{equation}
                \mathfrak{B}_f(\nu) = \sup_{\phi \in \mathcal{K}_f^{|\nu|}} \int_V \langle \phi, \mathrm{d} \nu \rangle,
        \end{equation}
        where $\nu \in \mathcal{M}(V, \R^{m+1})$ and where
        \begin{equation}
                \mathcal{K}_f^{|\nu|} := \{ (\phi^{\xi}, \phi^{\lambda}) \in L_{|\nu|}^{\infty}(V, \R^m \times \R) : f^*(t, \phi^{\xi}(t)) + \phi^{\lambda} \leq 0 \text{ for a.e. } t \in V \}
        \end{equation}
        is the set of \textup{essentially bounded measurable selections} (as opposed to the set $\mathcal{K}_f$ of continuous selections used in Definition~\ref{def:generalized_bb}).
        \label{lem:ess_bd_slct}
\end{lem}
\begin{proof}
        Since $\mathcal{K}_f \subset \mathcal{K}_f^{|\nu|}$, it suffices to show that
        \begin{equation}
                \mathfrak{B}_f(\nu) \geq \int_V \langle \phi, \mathrm{d} \nu \rangle \quad \forall \phi \in \mathcal{K}_f^{|\nu|}.
                \label{eq:bb_geq_bd_slct}
        \end{equation}
        In order to do so, we fix arbitrary $\phi = (\phi^{\xi}, \phi^{\lambda}) \in \mathcal{K}_f^{|\nu|}$ as well as $\epsilon > 0$ and construct~$\tilde{\phi} = (\tilde{\phi}^{\xi}, \tilde{\phi}^{\lambda}) \in \mathcal{K}_f$ with
        \begin{equation}
                \left\| \int_V \left( \phi - \tilde{\phi} \right) \d |\nu| \right\| \leq \epsilon,
        \end{equation}
        so that \eqref{eq:bb_geq_bd_slct} follows from the definition of $\mathfrak{B}_f$ as the supremum over $\mathcal{K}_f$.
        
        We begin by defining
        \begin{equation}
                M_f := \sup_{t \in V, \|\xi\| < \|\phi\|_{\infty}} |f^*(t, \xi)| < \infty
        \end{equation}
        and $M := \max\{M_f, \|\phi\|_{\infty}\}$.
        According to Lusin's theorem, there exist a closed set $A \subset V$ with $|\nu|(V \setminus A) \leq \epsilon / (2 M)$ as well as continuous, compactly supported functions $(\hat{\phi}^{\xi}, \hat{\phi}^{\lambda})$ that agree with $(\phi^{\xi}, \phi^{\lambda})$ on $A$ and satisfy both $\| \hat{\phi}^{\xi} \|_{\infty} \leq \| \phi^{\xi} \|_{\infty}$ and $\| \hat{\phi}^{\lambda} \|_{\infty} \leq \| \phi^{\lambda} \|_{\infty}$.
        If we now define
        \begin{align}
                \tilde{\phi}^{\xi}(t) &:= \hat{\phi}^{\xi}(t),\\
                \tilde{\phi}^{\lambda}(t) &:= \min \left\{ \hat{\phi}^{\lambda}(t), -f^*\left(t, \hat{\phi}^{\xi}(t)\right) \right\},
        \end{align}
        then $\| \tilde{\phi} \|_{\infty} \leq M$ and $\tilde{\phi}^{\lambda}$ is continuous with $\phi^{\lambda}(t) = \tilde{\phi}^{\lambda}(t)$ on~$A$, so that $\tilde{\phi} \in \mathcal{K}_f$.
        Therefore, one has
        \begin{align}
                \left\| \int_V \left( \phi - \tilde{\phi} \right) \d |\nu| \right\| &= \left\| \int_{V \setminus A} \left( \phi - \tilde{\phi} \right) \d |\nu| \right\| \\
                &\leq |\nu|(V \setminus A) \| \phi - \tilde{\phi} \|_{\infty} \\
                &\leq |\nu|(V \setminus A) \left( \| \phi \|_{\infty} + \| \tilde{\phi} \|_{\infty} \right) \\
                &\leq |\nu|(V \setminus A) (2M) \leq \epsilon
        \end{align}
        due to $|\nu|(V \setminus A) \leq \epsilon / (2 M)$. \qed
\end{proof}

We now proceed with the proof of Proposition~\ref{prop:enforced_ac}.
\begin{proof}[of Proposition~\ref{prop:enforced_ac}]
        For $\nu = (E, \mu) \in \mathcal{M}(V, \R^{m} \times \R)$, assume there exists a measurable set $A$ with $\mu(A) = 0$ and $E(A) \neq 0$.
        Let $k \in \mathbb{N}$ be arbitrary and define $\psi_k^{\xi} := (k E(A)) \mathds{1}_A(x)$ and $\psi_k^{\lambda} := -f^*(x, \psi_k^{\xi}(x)) \mathds{1}_A(x)$.
        
        Clearly, $(\psi_k^{\xi}, \psi_k^{\lambda}) \in \mathcal{K}_f^{|\nu|}$ and thus Lemma~\ref{lem:ess_bd_slct} implies
        \begin{align}
                \mathfrak{B}_f((E, \mu)) &= \sup_{(\phi^{\xi}, \phi^{\lambda}) \in \mathcal{K}_f^{|\nu|}} \int_V \langle \phi^{\xi}, \mathrm{d} E \rangle + \int_V \phi^{\lambda} \d \mu \\
                &\geq \int_V \langle \psi_k^{\xi}, \mathrm{d} E \rangle + \int_V \psi_k^{\lambda} \d \mu
                = k \| E(A) \|^2.
        \end{align}
        As $k$ was arbitrary, this yields $\mathfrak{B}_f((E, \mu)) = +\infty$. \qed
\end{proof}

While Proposition~\ref{prop:enforced_ac} covers a large number of integrands $f$, we emphasize that it still exempts an important class, namely those that are 1-homogeneous in the last argument.
Since the conjugation of a 1-homogeneous function always yields an indicator function (cf. \cite[Ex.~11.4]{rockafellar2004}), $f^*$ is neither continuous nor bounded.
As a concrete example, consider $f(t, p) = \| p \|$ with $f^*(t, \xi) = \delta_{ \{ \| \cdot \| \leq 1 \} }(\xi)$.
This yields
\begin{equation}
        \mathfrak{B}_f ((E, \mu)) = \sup_{\| \phi^{\xi} \|_{\infty} \leq 1} \int_V \langle \phi^{\xi}, \mathrm{d} E \rangle = \|E\|(V) < \infty
\end{equation}
independently of $\mu \in \mathcal{M}(V, \R)$, so that no conclusion about absolute continuity is allowed in this case.

\subsection{Lifting the Augmentation Constraint}
\label{subsec:constraint_lifting}

The results of the previous section imply the equivalence
\begin{equation}
        \inf_{u \in \mathcal{U}} \tilde{F}(u, p) \ \text{s.t.} \ p = Lu \quad = \quad \inf_{u \in \mathcal{U}} \mathfrak{B}_f((p \delta_u, \delta_u)) \ \text{s.t.} \ p = Lu
        \label{eq:bb_lifting_equiv}
\end{equation}
between the minimization of the nonconvex augmented functional $\tilde{F}$ and the convex Benamou-Brenier functional $\mathfrak{B}_f$.
Note that the lifted right-hand side problem cannot be phrased easily without the unlifted variables $u$ and $p$ due to the augmentation constraint $p = L u$.
Since we pursue convex relaxations that solely work on measure-valued variables, we propose in this section necessary conditions for such constraints in terms of the variables $\mu = \delta_u$ and $E = p \delta_u$.

A key observation is that the form of these conditions has to depend on the linear differential operator~$L$.
We will present conditions for the three cases of~${L = \nabla}$, $L = \nabla^2$ and $L = \Delta$.

\subsubsection{First-Order Condition}
This section covers the case of $L = \nabla$, i.e., $p = \nabla u$, and therefore $\R^m = \R^{d, s}$ in the definition of the integrand~$f$.
Given a pair of measures~${(E, \mu) \in \mathcal{M}(U, \R^{d, s} \times \R)}$ with $\mu = \delta_u$ for a differentiable function $u$, we will show that in order for $E$ to ``represent'' the gradient $\nabla u$, the \textit{first-order continuity equation}
\begin{equation}
        \nabla_x \mu + \ddiv_z E = 0
        \label{eq:1st_order_ce}
\end{equation}
has to be satisfied in the following sense:

\begin{defn}
        A pair of measures $(E, \mu) \in \mathcal{M}(U, \R^{d, s} \times \R)$ is said to satisfy the \textup{first-order continuity equation} \eqref{eq:1st_order_ce} if 
        \begin{equation}
                \int_U \ddiv_x \phi \d \mu + \int_U \langle \nabla_z \phi^{\top}, \mathrm{d} E \rangle = 0
                \label{eq:1st_order_ce_def}
        \end{equation}
        holds for all differentiable test functions $\phi$ with compact support in the first variable\footnote{We denote derivatives with respect to variables from~$\Omega$ by $x$ and derivatives with respect to variables from~$\Gamma$ by $z$.} as denoted by $C_c^1(\Omega \times \Gamma, \R^d)$.
        \label{def:1st_order_ce}
\end{defn}

The precise statement of the alluded necessary condition reads:

\begin{prop}
        Let $u \in C^1(\Omega, \Gamma)$ and $(E, \mu) \in \mathcal{M}(U, \R^{d, s} \times \R)$ with $\mu = \delta_u$ and $E = \nabla u \delta_u$ in the sense of
        \begin{equation}
                \int_U \langle \phi, \mathrm{d} E \rangle = \int_{\Omega} \langle \phi(x, u(x)), \nabla u(x) \rangle \d x
                \label{eq:E_nabla_u_def}
        \end{equation}
        for all $\phi \in C_0(\Omega \times \Gamma, \R^{d, s})$, then $(E, \mu)$ solves the first-order continuity equation as described in Definition~\ref{def:1st_order_ce}.
        \label{prop:1st_order_ce_nc}
\end{prop}
\begin{proof}
        Consider an arbitrary test function $\phi \in C_c^1(\Omega \times \Gamma, \R^d)$.
        Applying the chain rule yields
        \begin{equation}
                \ddiv_x[\phi(x, u(x))] = \ddiv_x \phi(x, u(x)) + \langle \nabla_z \phi(x, u(x))^{\top}, \nabla u(x) \rangle,
        \end{equation}
        where $\ddiv_x[\phi(x, u(x))]$ and $\ddiv_x \phi(x, u(x))$ refer to total and partial differentials with respect to $x$, respectively.
        This implies
        \begin{align}
                \int_U \ddiv_x \phi \d \mu &= \int_{\Omega} \ddiv_x \phi(x, u(x)) \d x \\
                &= \int_{\Omega} \ddiv_x[\phi(x, u(x))] \d x - \int_{\Omega} \langle \nabla_z \phi(x, u(x))^{\top}, \nabla u(x) \rangle \d x,
        \end{align}
        where the first summand evidently vanishes due to the divergence theorem and the compact support of $\phi$.
        The latter term equals $\int_U \langle \nabla_z \phi^{\top}, \mathrm{d} E \rangle$ through\eqref{eq:E_nabla_u_def}, so that the overall equation can be reordered to yield~\eqref{eq:1st_order_ce}.  \qed
\end{proof}

In fact, the necessary condition from Proposition~\ref{prop:1st_order_ce_nc} can be extended to a sufficient condition under suitable circumstances.

\begin{prop}
        For $u \in C^1(\Omega, \Gamma)$, $\mu = \delta_u$ and $E \in \mathcal{M}(U, \R^{d, s})$ with $E \ll \mu$, the continuity equation~\eqref{eq:1st_order_ce_def} for $(E, \mu)$ implies $E = \nabla u \delta_u$ almost everywhere on the support of $\mu$.
        \label{prop:1st_order_ce_sc}
\end{prop}
\begin{proof}
        Let $E = \nabla u \delta_u$ and let $\tilde{E} \in \mathcal{M}(U, \R^{d, s})$ with $\tilde{E} \ll \mu$ and $\tilde{E} = w \mu$ with~${w \in L_{\mu}^1(U, \R^{d, s})}$ solve the continuity equation~\eqref{eq:1st_order_ce_def} for $(\tilde{E}, \mu)$.
        This implies
        \begin{equation}
                \int_U \langle \nabla_z \phi^{\top}, w - \nabla u \rangle \d \mu = 0 \quad \forall \phi \in C_c^1(U, \R^d)
                \label{eq:1st_order_ce_sc_eq1}
        \end{equation}
        as $(E, \mu)$ solves~\eqref{eq:1st_order_ce_def} as well by Proposition~\ref{prop:1st_order_ce_nc}.
        Consider now $\psi \in C_c^1(\Omega, \R^{d, s})$ and construct from it $\hat{\phi}(x, z) := \psi(x) z$.
        As $\hat{\phi} \in C_c^1(U, \R^d)$ with $\nabla_z \hat{\phi}^{\top} = \psi$, one has
        \begin{equation}
                \int_{\Omega} \langle \psi(x), w(x, u(x)) - \nabla u(x) \rangle \d \mu = 0
        \end{equation}
        from \eqref{eq:1st_order_ce_sc_eq1}.
        Therefore, $w = \nabla u$ holds almost everywhere on the support of $\mu$. \qed
\end{proof}

The results of this section allow to define a lifted version $\mathcal{F}$ of the original problem $\inf_{u \in C^1(\Omega, \Gamma)} \int_{\Omega} f(x, u(x), \nabla u(x)) \d x$ over $\mathcal{M}(\Omega \times \Gamma)$ -- or, respectively, suitable subsets thereof -- as 
\begin{equation}
        \mathcal{F}(\mu) := \inf \{ \mathfrak{B}_f((E, \mu)) : (E, \mu) \text{ solves } \eqref{eq:1st_order_ce_def} \}.
        \label{eq:1st_order_lifting}
\end{equation}
Taken together, the results of Corollary~\ref{cor:bb_lifting} and Propositions~\ref{prop:enforced_ac},~\ref{prop:1st_order_ce_nc},~\ref{prop:1st_order_ce_sc} imply the central result
\begin{equation}
        \mathcal{F}(\delta_u) = F(u) := \int_{\Omega} f(x, u(x), \nabla u(x)) \d x
        \label{eq:1st_order_equiv}
\end{equation}
for $u \in C^1(\Omega, \Gamma)$ with $F(u) < \infty$, given that all relevant assumptions hold.

At last, a convex lifting model can be obtained by choosing the admissible set for~\eqref{eq:1st_order_lifting} as the set of weakly measurable functions~$L_w^{\infty}(\Omega, \mathcal{P}(\Gamma))$ which was introduced in Section~\ref{subsec:preliminaries}.
The same reasoning applies to the second-order models~\eqref{eq:2nd_order_lifting} and~\eqref{eq:laplacian_lifting} that we describe in following two subsections. 

\subsubsection{Second-Order Condition}
This section will cover the case of $L = \nabla^2$ in a similar fashion to the previously discussed first-order case.
We regard the Hessian of $u \in C^2(\Omega, \Gamma)$ as a third-order tensor, i.e., $\R^m = \R^{d, d, s}$, as one obtains a $d \times d$-Hessian matrix $\nabla^2 u_k$ for every component $k = 1, \ldots, s$.

Analogously to the first-order continuity equation, we introduce a \textit{second-order continuity equation} which needs to be satisfied in order for $E$ to ``represent'' the Hessian $\nabla^2 u$.
This equation reads
\begin{equation}
        - \nabla_x^2 \mu - \ddiv_z E + \ddiv_z^2 H = 0
        \label{eq:2nd_order_ce}
\end{equation}
for an additional auxiliary variable $H \in \mathcal{M}(U, \R^{d, s, d, s})$ ``corresponding'' to the tensor product $\nabla u \otimes \nabla u$.
As before, we define solutions of \eqref{eq:2nd_order_ce} in a weak sense:

\begin{defn}
        A triple $(H, E, \mu) \in \mathcal{M}(U, \R^{d, s, d, s} \times \R^{d, d, s} \times \R)$ is said to be a solution of the \textup{second-order continuity equation}~\eqref{eq:2nd_order_ce} if
        \begin{equation}
                - \int_U \ddiv_x^2 \phi \d \mu + \int_U \langle \nabla_z \phi, \mathrm{d} E \rangle + \int_U \langle \nabla_z^2 \phi, \mathrm{d} H \rangle = 0
        \end{equation}
        for all $\phi \in C_c^2(U, \R^{d, d})$ holds\footnote{In order to avoid notational overhead, we omit a denotation of the correct ``transposition'' of the third- and fourth-order tensors $\nabla_z \phi$ and $\nabla_z^2 \phi$ and simply assume them to be ordered in the appropriate format.}.
        \label{def:2nd_order_ce}
\end{defn}

The analogous necessary condition to the one introduced in Proposition~\ref{prop:1st_order_ce_nc} for the first-order case is as follows:

\begin{prop}
        Let $u \in C^2(\Omega, \Gamma)$, $(H, E, \mu) \in \mathcal{M}(U, \R^{d, s, d, s} \times \R^{d, d, s} \times \R)$ with $\mu = \delta_u$, $E = (\nabla^2 u) \delta_u$ and $H = (\nabla u \otimes \nabla u) \delta_u$ in the sense of
        \begin{equation}
                \int_U \langle \phi, \mathrm{d} H \rangle := \int_{\Omega} \langle \phi(x, u(x)), (\nabla u \otimes \nabla u) (x) \rangle \d x
        \end{equation}
        for all $\phi \in C_0(U, \R^{d, s, d, s})$.
        Then, $(H, E, \mu)$ solves the second-order continuity equation as given in Definition~\ref{def:2nd_order_ce}.
        \label{prop:2nd_order_ce_nc}
\end{prop}
\begin{proof}
        Let $h : \R^d \to \R^d \times \R^s$ denote the mapping of $x \in \R^d$ onto $(x, u(x))$.
        For given $\phi \in C_c^2(U, \R^{d, d})$, we will consider the expression~${\ddiv_x^2 \phi \circ h \big\vert_{x_0}}$ that is explained by
        \begin{equation}
                \left. \ddiv_x
                \begin{pmatrix}
                        \ddiv_x \phi_1 \circ h \\
                        \vdots \\
                        \ddiv_x \phi_d \circ h \\
                \end{pmatrix} \right\vert_{x_0}
                \quad \text{for} \quad
                \phi =
                \begin{pmatrix}
                        \phi_1 \\
                        \vdots \\
                        \phi_d
                \end{pmatrix} =
                \begin{pmatrix}
                \phi_{1, 1} &\ldots &\phi_{1, d} \\
                \vdots & & \vdots \\
                \phi_{d, 1} &\ldots &\phi_{d, d}
                \end{pmatrix}.
                \label{eq:2nd_order_ce_nc_eq1}
        \end{equation}
        For each $i = 1, \ldots, d$, one has
        \begin{align}
                \ddiv_x \phi_i \circ h \big\vert_x &= \sum_{j=1}^d \partial_{x_j} \phi_{i, j} \circ h \big\vert_x = \sum_{j=1}^d \langle \nabla \phi_{i, j} \big\vert_{h(x)}, \partial_{x_j} h \big\vert_x \rangle\\
                &= \sum_{j=1}^d \partial_{x_j} \phi_{i, j} \big\vert_{h(x)} + \sum_{k=1}^s \partial_{z_k} \phi_{i, j} \big\vert_{h(x)} \partial_{x_j} u_k \big\vert_x, \label{eq:2nd_order_ce_nc_eq2}
        \end{align}
        where the first summand of \eqref{eq:2nd_order_ce_nc_eq2} equals $\ddiv_x \phi_i \big\vert_{h(x)}$.
        Inserting this intermediate result into \eqref{eq:2nd_order_ce_nc_eq1} yields
        \begin{equation}
                \ddiv_x^2 \phi \circ h \big\vert_{x_0} =
                \sum_{i=1}^d \underbrace{\partial_{x_i} (\ddiv_x \phi_i) \circ h \big\vert_{x_0}}_{(A)} +
                \sum_{i,j=1}^d \sum_{k=1}^s \underbrace{\partial_{x_i} ((\partial_{z_k} \phi_{i,j}) \circ h) (\partial_{x_j} u_k) \big\vert_{x_0}}_{(B)}.
        \end{equation}
        The term $(A)$ can further be evaluated to
        \begin{align}
                (A) &= \langle \nabla (\ddiv_x \phi_i) \big\vert_{h(x_0)}, \partial_{x_i} h \big\vert_{x_0} \rangle \\
                &= \partial_{x_i} (\ddiv_x \phi_i) \big\vert_{h(x_0)} +
                \sum_{l=1}^s \partial_{z_l} (\ddiv_x \phi_i) \big\vert_{h(x_0)} \partial_{x_i} u_l \big\vert_{x_0} \\
                &= \sum_{i=1}^d \partial_{x_i} \partial_{x_j} \phi_{i, j} \big\vert_{h(x_0)} +
                \sum_{j=1}^d \sum_{l=1}^s \partial_{z_l} \partial_{x_j} \phi_{i, j} \big\vert_{h(x_0)} \partial_{x_i} u_l \big\vert_{x_0}.
        \end{align}
        For $(B)$, the product rule has to be applied, so that
        \begin{equation}
                (B) = \underbrace{\partial_{x_i} (\partial_{z_k} \phi_{i, j}) \circ h \big\vert_{x_0}}_{(C)} \partial_{x_j} u_k \big\vert_{x_0}
                + (\partial_{z_k} \phi_{i, j}) \circ h \big\vert_{x_0} \partial_{x_i} \partial_{x_j} u_k \big\vert_{x_0}.
        \end{equation}
        The expression $(C)$ equals
        \begin{align}
                (C) &= \langle \nabla (\partial_{z_k} \phi_{i, j}) \big\vert_{h(x_0)}, \partial_{x_i} h \big\vert_{x_0} \rangle \\
                &= \partial_{x_i} \partial_{z_k} \phi_{i, j} \big\vert_{h(x_0)} + \sum_{l=1}^s \partial_{z_l} \partial_{z_k} \phi_{i, j} \big\vert_{h(x_0)} \partial_{x_i} u_l\big\vert_{x_0},
        \end{align}
        so that one has
        \begin{align}
                \ddiv_x^2 \phi \circ h \big\vert_{x_0} &=
                \sum_{i,j=1}^d \partial_{x_i} \partial_{x_j} \phi_{i, j} \big\vert_{h(x_0)} \label{eq:2nd_order_ce_nc_eq3} \\
                &+ \sum_{i,j=1}^d \sum_{k=1}^s \partial_{x_j} \partial_{z_k} (\phi_{i, j} + \phi_{j, i}) \big\vert_{h(x_0)} \partial_{x_i} u_k \big\vert_{x_0} \label{eq:2nd_order_ce_nc_eq4} \\
                &+ \sum_{i,j=1}^d \sum_{k=1}^s \partial_{z_k} \phi_{i, j} \big\vert_{h(x_0)} \partial_{x_i} \partial_{x_j} u_k \big\vert_{x_0} \label{eq:2nd_order_ce_nc_eq5} \\
                &+ \sum_{i,j=1}^d \sum_{k, l=1}^s \partial_{z_l} \partial_{z_k} \phi_{i, j} \big\vert_{h(x_0)} \partial_{x_i} u_l \big\vert_{x_0} \partial_{x_j} u_k \big\vert_{x_0}. \label{eq:2nd_order_ce_nc_eq6}
        \end{align}
        Now, consider the summand~\eqref{eq:2nd_order_ce_nc_eq4}.
        Proceeding in a similar manner as above, one obtains the equivalence
        \begin{align}
                & \sum_{i,j=1}^d \sum_{k=1}^s \partial_{x_j} \partial_{z_k} (\phi_{i, j} + \phi_{j, i}) \big\vert_{h(x_0)} \partial_{x_i} u_k \big\vert_{x_0} \label{eq:2nd_order_ce_nc_eq7} \\
                =& \sum_{i,j=1}^d \sum_{k=1}^s \partial_{x_j} (\partial_{z_k} (\phi_{i, j} + \phi_{j, i}) \circ h) \big\vert_{x_0} \partial_{x_i} u_k \big\vert_{x_0} \label{eq:2nd_order_ce_nc_eq8} \\
                & -2 \sum_{i,j=1}^d \sum_{k, l=1}^s \partial_{z_l} \partial_{z_k} \phi_{i, j} \big\vert_{h(x_0)} \partial_{x_i} u_l \big\vert_{x_0} \partial_{x_j} u_k \big\vert_{x_0}. \label{eq:2nd_order_ce_nc_eq9}
        \end{align}
        Inserting \eqref{eq:2nd_order_ce_nc_eq7}--\eqref{eq:2nd_order_ce_nc_eq9} into \eqref{eq:2nd_order_ce_nc_eq3}--\eqref{eq:2nd_order_ce_nc_eq6} yields
        \begin{align}
                \ddiv_x^2 \phi \circ h \big\vert_{x_0} &=
                \sum_{i,j=1}^d \partial_{x_i} \partial_{x_j} \phi_{i, j} \big\vert_{h(x_0)} \label{eq:2nd_order_ce_nc_eq10} \\
                &+ \sum_{i,j=1}^d \sum_{k=1}^s \partial_{x_j} (\partial_{z_k} (\phi_{i, j} + \phi_{j, i}) \circ h) \big\vert_{x_0} \partial_{x_i} u_k \big\vert_{x_0} \label{eq:2nd_order_ce_nc_eq11} \\
                &+ \sum_{i,j=1}^d \sum_{k=1}^s \partial_{z_k} \phi_{i, j} \big\vert_{h(x_0)} \partial_{x_i} \partial_{x_j} u_k \big\vert_{x_0} \label{eq:2nd_order_ce_nc_eq12} \\
                &- \sum_{i,j=1}^d \sum_{k, l=1}^s \partial_{z_l} \partial_{z_k} \phi_{i, j} \big\vert_{h(x_0)} \partial_{x_i} u_l \big\vert_{x_0} \partial_{x_j} u_k \big\vert_{x_0}. \label{eq:2nd_order_ce_nc_eq13}
        \end{align}
        When integrating \eqref{eq:2nd_order_ce_nc_eq11} and \eqref{eq:2nd_order_ce_nc_eq12} over $\Omega$, one finds due to Gauss' theorem and due to the compact support of $\phi$ that
        \begin{align}
                &\sum_{i=1}^d \sum_{k=1}^s \int_{\Omega} \ddiv_x (\partial_{z_k} \phi_{i}) \circ h \big\vert_{x} \partial_{x_i} u_k \big\vert_{x} \d x \\
                + &\sum_{i=1}^d \sum_{k=1}^s \int_{\Omega} \left\langle (\partial_{z_k} \phi_{i}) \circ h \big\vert_{x} , \nabla (\partial_{x_j} u_k) \big\vert_{x} \right\rangle \d x \\
                = & \sum_{i=1}^d \sum_{k=1}^s \int_{\partial \Omega} \left\langle \left(\partial_{x_i} u_k \big\vert_{x}\right) \left((\partial_{z_k} \phi_{i}) \circ h \big\vert_{x}\right) , n(x) \right\rangle \d x = 0,
        \end{align}
        where $n$ is the outer normal of $\Omega$.
        For the remaining part of~\eqref{eq:2nd_order_ce_nc_eq11}, similar arguments produce
        \begin{align}
                &\sum_{i=1}^d \sum_{k=1}^s \int_{\Omega} \sum_{j=1}^d \partial_{x_j} ((\partial_{z_k} \phi_{j, i}) \circ h) \big\vert_{x} \partial_{x_i} u_k \big\vert_{x} \d x \\
                =  -&\sum_{i=1}^d \sum_{k=1}^s \int_{\Omega} \sum_{j=1}^d (\partial_{z_k} \phi_{i, j}) \circ h \big\vert_{x} \partial_{x_i} \partial_{x_j} u_k \big\vert_{x} \d x.
        \end{align}
        With this, the complete integral of \eqref{eq:2nd_order_ce_nc_eq10} -- \eqref{eq:2nd_order_ce_nc_eq13} over $\Omega$ evaluates to
        \begin{align}
                0 = &\int_{\Omega} \ddiv_x^2 \phi \circ h \big\vert_{x} \d x \\
                 = &\int_{\Omega} \ddiv_x^2 \phi \big\vert_{h(x)} \d x -\int_{\Omega} \left\langle \nabla_z \phi \big\vert_{h(x)}, \nabla^2 u \big\vert_{x} \right\rangle \d x \\
                 &-\int_{\Omega} \left\langle \nabla_z^2 \phi \big\vert_{h(x)}, \nabla u \otimes \nabla u \big\vert_{x} \right\rangle \d x
        \end{align}
        due to the compact support of $\phi$, which concludes the proof as the right hand side equals the definition of the second-order continuity equation in~\eqref{eq:2nd_order_ce}.  \qed
\end{proof}

As before, the necessary condition from Proposition~\ref{prop:2nd_order_ce_nc} can be extended to a sufficient one in the case of absolute continuity.

\begin{prop}
        For $u \in C^2(\Omega, \Gamma)$, let $(H, E, \mu) \in \mathcal{M}(U, \R^{d, s, d, s} \times \R^{d, d, s} \times \R)$ with $\mu = \delta_u$ solve the second-order continuity equation~\eqref{eq:2nd_order_ce}.
        Then, $E \ll \mu$ implies $E = (\nabla^2 u) \delta_u$ almost everywhere on the support of $\mu$.
        
        If, additionally, $H \ll \mu$ holds with $H = A \mu$ such that $[A(x)]_{i, \cdot, j, \cdot}$ is symmetric for all $i, j = 1, \ldots, d$ and $x \in U$, then $A = \nabla u \otimes \nabla u$ almost everywhere on the support of $\mu$.
        \label{prop:2nd_order_ce_sc}
\end{prop}
\begin{proof}
        Let $E = \nabla^2 u$, $H = \nabla u \otimes \nabla u$ and let $(\tilde{H}, \tilde{E}) \in \mathcal{M}(U, \R^{d, s, d, s} \times \R^{d, d, s})$ with~${\tilde{E} \ll \mu}$ and $\tilde{E} = w \mu$ with $w \in L^1_{\mu}(\R^{d, d, s})$ be such that $(\tilde{H}, \tilde{E}, \mu)$ solves the second-order continuity equation~\eqref{eq:2nd_order_ce}.
        This implies
        \begin{equation}
                \int_U \langle \nabla_z \phi, w - \nabla^2 u \rangle \d \mu + \int_U \langle \nabla_z^2 \phi , \d (\tilde{H} - H) \rangle = 0 \quad \forall \phi \in C_c^2(U, \R^{d, d}). \label{eq:2nd_order_ce_sc_eq1}
        \end{equation}
        Let now $\psi \in C_c^2(\Omega, \R^{d,d,s})$ and construct from it $\hat{\phi}_{i,j}(x, z) := \langle \psi_{i,j}(x), z \rangle$.
        Obviously, $\hat{\phi} \in C_c^2(U, \R^{d,d})$ holds with $\nabla_z \hat{\phi} = \psi$ and $\nabla_z^2 \phi = 0$.
        Inserting $\hat{\phi}$ into~\eqref{eq:2nd_order_ce_sc_eq1} yields
        \begin{equation}
                \int_{\Omega} \langle \psi(x), w(x, u(x)) - \nabla^2 u(x) \rangle \d x = 0,
        \end{equation}
        which is why $w = \nabla^2 u$ holds almost everywhere on the support of $\mu$.
        
        As for the additional claim, let $\tilde{H} = \tilde{A} \mu$ be symmetric in the given sense.
        Furthermore, let $\psi \in C_c^2(\Omega, \R^{d,s,d,s})$ be an arbitrary symmetric function in the same sense, i.e., $\psi_{i,k,j,l}(x) = \psi_{i,l,j,k}(x)$ for~${i, j \in \{ 1, \ldots, d \} }$,~${k, l \in \{ 1, \ldots, s \} }$ and~${x \in \Omega}$.
        Now construct $\hat{\phi} \in C_c^2(U, \R^{d, d})$ through~${\hat{\phi}_{i, j}(x, z) := z^{\top} \psi_{i,\cdot,j,\cdot}(x) z}$ for $i, j = 1, \ldots, d$.
        Clearly, $\nabla_z^2 \hat{\phi} = \psi$ holds and therefore
        \begin{equation}
                \int_{U} \langle \psi, \d (\tilde{H} - H) \rangle = \int_{\Omega} \langle \psi(x, u(x)) , \tilde{A}(x, u(x)) - (\nabla u \otimes \nabla u)(x) \rangle \d x.
        \end{equation}
        By the arbitrariness (and symmetry) of $\psi$, one can conclude that $\tilde{A} = (\nabla u \otimes \nabla u)$ is true almost everywhere on the support of $\mu$. \qed
\end{proof}

As a final result for the second-order case, we define a lifting of the problem~${\inf_{u \in C^2(\Omega, \Gamma)} \int_{\Omega} f(x, u(x), \nabla^2 u(x)) \d x}$ over $\mathcal{M}(\Omega \times \Gamma)$ as
\begin{equation}
        \mathcal{F}(\mu) := \inf \{ \mathfrak{B}_f((E, \mu)) : (H, E, \mu) \text{ solves } \eqref{eq:2nd_order_ce} \text{ for symmetric } H \}.
        \label{eq:2nd_order_lifting}
\end{equation}
Note that although the results of the absolute continuity criterion in Proposition~\ref{prop:enforced_ac} do not apply to the additional condition from Proposition~\ref{prop:2nd_order_ce_sc}~--~the former is only concerned with the question of $E \ll \mu$ and not $H \ll \mu$ --, we can still formulate an analogue to \eqref{eq:1st_order_equiv} in the second-order case.
Precisely, one can conclude from Corollary~\ref{cor:bb_lifting} and Propositions~\ref{prop:enforced_ac}, \ref{prop:2nd_order_ce_nc}, \ref{prop:2nd_order_ce_sc} that 
\begin{equation}
        \mathcal{F}(\delta_u) = F(u) := \int_{\Omega} f(x, u(x), \nabla^2 u(x)) \d x
        \label{eq:2nd_order_equiv}
\end{equation}
holds for $u \in C^2(\Omega, \Gamma)$ with $F(u) < \infty$, supposing that all relevant assumptions on $f$ are met. 

\subsubsection{Laplacian Condition}
As an important class of second-order models, this section discusses the case $L = \Delta$, i.e. $p = \Delta u$, where the Laplacian of a vector-valued function $u: \Omega \to \Gamma \subset \R^s$ is understood in a componentwise sense, so that $\R^m = \R^s$.
Put differently, this case is derived from the previously discussed (full) second-order model by only considering the trace of each Hessian~${\nabla^2 u_k}$ for $k = 1, \ldots, s$.

Consequently, the second-order continuity equation~\eqref{eq:2nd_order_ce} simplifies to
\begin{equation}
        -\Delta_x \mu - \ddiv_z E + \ddiv_z^2 H = 0
        \label{eq:laplacian_ce}
\end{equation}
for vectorial measures $(H, E) \in \mathcal{M}(U, \R^{s, s} \times \R^s)$ of reduced dimensionality  (when compared to the full model).
Solutions to \eqref{eq:laplacian_ce} are defined as follows:

\begin{defn}
        A triple $(H, E, \mu) \in \mathcal{M}(U, \R^{s, s} \times \R^s \times \R)$ of measures is said to satisfy the \textup{Laplacian continuity equation}~\eqref{eq:laplacian_ce} if 
        \begin{equation}
                - \int_U \Delta_x \phi \d \mu + \int_U \langle \nabla_z \phi , \d E \rangle + \int_U \langle \nabla_z^2 \phi , \d H\rangle = 0
                \label{eq:laplacian_ce_def}
        \end{equation}
        holds for all $\phi \in C_c^2(U, \R)$.
        \label{def:laplacian_ce}
\end{defn}

Unsurprisingly, analogous necessary and sufficient conditions to the ones from the full second-order case, i.e., Propositions~\ref{prop:2nd_order_ce_nc} and \ref{prop:2nd_order_ce_sc}, hold true for the Laplacian continuity equation.
As their corresponding proofs proceed completely analogously to the ones from the previous section, we omit them here.

\begin{prop}
        Let $u \in C^2(\Omega, \Gamma)$ and let $(H, E, \mu) \in \mathcal{M}(U, \R^{s, s} \times \R^s \times \R)$ with~$\mu = \delta_u$, $E = (\Delta u) \delta_u$ and $H = (\sum_{i=1}^d \partial_{x_i} u \otimes \partial_{x_i} u) \delta_u$ in the sense of
        \begin{equation}
                \int_U \langle \phi, \d H \rangle := \int_{\Omega} \left\langle \phi(x, u(x)), \sum\nolimits_{i=1}^d (\partial_{x_i} u \otimes \partial_{x_i} u) (x) \right\rangle \d x
        \end{equation}
        for all $\phi \in C_0(U, \R^{s, s})$.
        Then, $(H, E, \mu)$ solves the Laplacian continuity equation as defined in \eqref{eq:laplacian_ce_def}.
        \label{prop:laplacian_ce_nc}
\end{prop}

\begin{prop}
        Let $u \in C^2(\Omega, \Gamma)$ and let~${(H, E, \mu) \in \mathcal{M}(U, \R^{s, s} \times \R^s \times \R)}$ for $\mu = \delta_u$ solve the Laplacian continuity equation~\eqref{eq:laplacian_ce_def}.
        Then, $E \ll \mu$ implies~${E = (\Delta u) \delta_u}$ almost everywhere on the support of $\mu$.
        
        If, additionally, $H \ll \mu$ holds for a symmetric density $A$ of $H = A \delta_u$, then~${H = (\sum_{i=1}^d \partial_{x_i} u \otimes \partial_{x_i} u) \delta_u}$ is true almost everywhere on the support of $\mu$.
        \label{prop:laplacian_ce_sc}
\end{prop}

Similarly to the previous two cases, we define a lifted version of the problem~${\inf_{u \in C^2(\Omega, \Gamma)} \int_{\Omega} f(x, u(x), \Delta u(x)) \d x}$ over $\mathcal{M}(\Omega \times \Gamma)$ as
\begin{equation}
        \mathcal{F}(\mu) := \inf\{ \mathfrak{B}_f((E, \mu)) : (H, E, \mu) \text{ solves } \eqref{eq:laplacian_ce_def} \text{ for sym. pos. semidef. } H \}.
        \label{eq:laplacian_lifting}
\end{equation}
Here, we added the additional assumption of positive semidefiniteness on $H$ as this property is always satisfied under the (full) assumptions of Proposition~\ref{prop:2nd_order_ce_sc} and as it therefore constitutes a reasonable restriction of the feasible set for the lifted functional~$\mathcal{F}$.

As before, we obtain the following relationship between the original problem and the lifted one thanks to Corollary~\ref{cor:bb_lifting} and Propositions~\ref{prop:enforced_ac},~\ref{prop:laplacian_ce_nc},~\ref{prop:laplacian_ce_sc}:
\begin{equation}
        \mathcal{F}(\delta_u) = F(u) := \int_{\Omega} f(x, u(x), \Delta u(x)) \d x
        \label{eq:laplacian_lifting_compatibility}
\end{equation}
for $u \in C^2(\Omega, \Gamma)$ with $F(u) < \infty$, supposing that the integrand $f$ is as in Corollary~\ref{cor:bb_lifting}. 
%

\subsection{Connections to Lifting Models and Optimal Transport Problems}
\label{subsec:connections}

\subsubsection{Scalar-valued Subgraph-Lifting \cite{pock2010}}
In this section, we establish a connection of our proposed first-order model~\eqref{eq:1st_order_lifting} with scalar range~$\Gamma$ to the subgraph-representation approach from~\cite{pock2008,pock2010} as recapitulated in Section~\ref{subsec:related_work}.
The key insight linking the two models is the fact that the latter, i.e.,~${\sup_{\phi \in \mathcal{K}} \int_{\Omega \times \Gamma} \langle \phi, D v \rangle}$ from \eqref{eq:calibration_lifting} is, in truth, not a functional in $v \in \BV(\Omega \times \Gamma)$ itself, but rather in its distributional derivative $D v \in \mathcal{M}(\Omega \times \Gamma, \R^d \times \R)$.

For integrands of the form 
\begin{equation}
        f(x, u(x), \nabla u (x)) = \rho(x, u(x)) + \eta(\nabla u (x))
\end{equation}
the dually admissible set~$\mathcal{K}_f$ \eqref{eq:bb_dual_set} of our proposed Benamou-Brenier functional proves to be equal to the admissible set~$\mathcal{K}$ \eqref{eq:calibration_dual_set} of the subgraph-lifting up to a change in the sign of the last component of the included test functions $\phi$.
We address this discrepancy by introducing the notation $D^- v$ for the vectorial measure which is equal to $D v$ in the first~$d$ components and that is the negative of $D v$ in the last component.
This allows for the compact denotation of the equivalence between the two models as
\begin{equation}
        \sup_{\phi \in \mathcal{K}} \int_{\Omega \times \Gamma} \langle \phi, D v \rangle = \mathfrak{B}_f (D^- v).
        \label{eq:subgraph_bb_equiv}
\end{equation}

In fact, one of the key results of \cite{pock2010} is obtained as a corollary from the derivations in this work:
\begin{cor}[{\cite[Theorem~3.2]{pock2010}}]
        Let $\Omega \subset \R^d$ be open, let $\Gamma \subset \R$ be compact and let~$f: \Omega \times \Gamma \times \R^d \to [0, \infty]$ be as in Theorem~\ref{thm:bb_integral_representation}, then
        \begin{equation}
                \int_{\Omega} f(x, u(x), \nabla u(x)) \d x = \sup_{\phi \in \mathcal{K}} \int_{\Omega \times \Gamma} \langle \phi, D \mathbf{1}_u \rangle
        \end{equation}
        holds for all $u \in W^{1,1}(\Omega)$.
\end{cor}
\begin{proof}
        As the assumptions of Theorem~\ref{thm:bb_integral_representation} are all satisfied, we infer from \eqref{eq:subgraph_bb_equiv} the equivalence
        \begin{equation}
                \sup_{\phi \in \mathcal{K}} \int_{\Omega \times \Gamma} \langle \phi, D v \rangle = \int_{\Omega \times \Gamma} \tilde{h}_f \left((x, z), \frac{\d (D^- \mathbf{1}_u)}{\d | D^- \mathbf{1}_u |} (x, z) \right) \, | D^- \mathbf{1}_u | (d (x, z)).
                \label{eq:pock_theorem32_eq1}
        \end{equation}
        Since $u \in W^{1,1} (\Omega)$, we know that $D^- \mathbf{1}_u$ is concentrated on the graph of~$u$ and that the explicit form of the vectorial density appearing in~\eqref{eq:pock_theorem32_eq1} is given by
        \begin{equation}
                \frac{\d (D^- \mathbf{1}_u)}{\d | D^- \mathbf{1}_u |} = \frac{(\nabla u, 1)}{\sqrt{1 + \| \nabla u \| ^ 2}}.
        \end{equation}
        The density of~$D^- \mathbf{1}_u$ with respect to $\delta_u$, on the other hand, is given simply by~$\d (D^- \mathbf{1}_u) / \d \delta_u = (\nabla u, 1)$.
        This implies, in particular,~$\mu = \delta_u$ for the last component~$\mu$ of~$D^- \mathbf{1}_u$.
        By the positive homogeneity of $\tilde{h}_f$, one obtains
        \begin{equation}
                \tilde{h}_f \left((x, z), \frac{\d (D^- \mathbf{1}_u)}{\d | D^- \mathbf{1}_u |} (x, z) \right) = \frac{ \tilde{h}_f ((x, z), (\nabla u (x), 1)) }{\sqrt{1 + \| \nabla u \| ^ 2}}
                \label{eq:pock_theorem32_eq2}
        \end{equation}
        for all $(x, z)$ on the graph of $u$.
        Bearing in mind that the denominator of~\eqref{eq:pock_theorem32_eq2} equals the density~$\d \mu / \d | D^- \mathbf{1}_u |$, equations~\eqref{eq:pock_theorem32_eq1} and~\eqref{eq:pock_theorem32_eq2} together yield
        \begin{align}
                \sup_{\phi \in \mathcal{K}} \int_{\Omega \times \Gamma} \langle \phi, D v \rangle &= 
                \int_{\Omega \times \Gamma} \tilde{h}_f ((x, z), (\nabla u (x), 1)) \d \mu \\
                &=\int_{\Omega \times \Gamma} f(x, z, \nabla u(x)) \d \delta_u,
        \end{align}
        which proves the assertion. \qed
\end{proof}

Since the two models' energy functionals are equal by~\eqref{eq:subgraph_bb_equiv}, differences can only be found in their respective admissible sets.
It is well-known that the left-hand side of~\eqref{eq:subgraph_bb_equiv} achieves finite energies only for functions~$v \in \mathcal{C}$ from~\eqref{eq:calibration_admissible_set} which are nonincreasing in the second argument.
Therefore, the admissible set~$\mathcal{C}$ may be amended by a corresponding constraint without loss of generality.
As a result, the last component of $D^- v$ can be seen to be in $L_w^{\infty}(\Omega, \mathcal{P}(\Gamma))$ -- the admissible set of our proposed model -- for all such $v$ by the slicing theory for functions of bounded variation \cite[Lemma~3.106]{ambrosio2000}.
The converse inclusion, however, i.e., the question whether for each $\mu \in L_w^{\infty}(\Omega, \mathcal{P}(\Gamma)$ there exists a~${v \in C}$ such that the last component of $D^- v$ equals $\mu$, does not allow for a positive answer as easily.
We resort to the observation that an additional degree of regularity in the spatial domain~$\Omega$ is required for~$\mu$ -- intuitively, this regularity is provided by the continuity equation~\eqref{eq:1st_order_ce} -- and leave a thorough analysis for future work.

To summarize, we found the scalar version of our proposed first-order model to be largely equivalent to the subgraph-based approach from~\cite{pock2010}.
At the same time, our methodology provides the advantage of a natural extension to vectorial ranges and higher orders of regularization, even though these generalizations come at the expense of a rounding procedure with guaranteed optimality for the unlifted problem as in~\cite[Theorem~3.1]{pock2010}. 

\subsubsection{Lifting Problems with Laplacian Regularization \cite{vogt2019a}}
The first continuous formulation of a lifting strategy for~\eqref{eq:target_functional} in case of a Laplacian regularization, i.e., for problems with integrands of the form
\begin{equation}
        f(x, u(x), \Delta u (x)) = \rho(x, u(x)) + \eta(\Delta u (x))
\end{equation}
was proposed in \cite{vogt2019a}.
In that work, the original problem is lifted to the functional
\begin{equation}
        \mathcal{F}(\mu) = \sup_{(p, q) \in X} \int_{\Omega \times \Gamma} (\Delta_x p + q) \d \mu
        \label{eq:laplace_lifting_ssvm}
\end{equation}
acting on measure-valued functions $\mu : \Omega \to \mathcal{P}(\Gamma)$.
Furthermore, the dually admissible vector fields~$(p, q)$ are given by the set
\begin{multline}
        X = \{ (p, q) \in C_c^2(\Omega \times \Gamma) \times L^1(\Omega \times \Gamma) : z \mapsto p(x, z) \text{ concave } \forall x \in \Omega, \\ q(x, z) + f^*(x, z, \nabla_z p(x, z)) \leq 0 \ \forall (x, z) \in \Omega \times \Gamma \}.
        \label{eq:laplace_lifting_dual_set_ssvm}
\end{multline}
It is then shown in \cite[Proposition~1]{vogt2019a} that
\begin{equation}
        \mathcal{F}(\delta_u) \leq \int_{\Omega} f(x, u(x), \Delta u (x)) \d x
        \label{eq:laplace_lifting_ssvm_ineq}
\end{equation}
holds for all sufficiently smooth~$u$.

As we shall show, the functional~\eqref{eq:laplace_lifting_ssvm} is closely connected to the Laplacian lifting method derived in this work.
To this end, consider the primal-dual formulation
\begin{multline}
        \inf_{E, H} \sup_{\substack{\phi \in \mathcal{K}_f, p}} \, \int_{\Omega \times \Gamma} (-\Delta_x p + \phi^{\lambda}) \d \mu + \int_{\Omega \times \Gamma} \langle \nabla_z p + \phi^{\xi}, \d E \rangle + \int_{\Omega \times \Gamma} \langle \nabla_z^2 p, \d H \rangle\\
        \text{ s.t. } H \text{ sym. pos. semidef.}
        \label{eq:laplacian_lifting_pd}
\end{multline}
of~\eqref{eq:laplacian_lifting} for given $\mu \in \mathcal{M}(\Omega \times \Gamma)$.
In accordance with definition~\ref{def:laplacian_ce}, the test functions $p$ are chosen from $C_c^2(\Omega \times \Gamma)$.
Formally swapping the order of minimization and maximization in~\eqref{eq:laplacian_lifting_pd} yields $\phi^{\xi} = -\nabla_z p$ as well as the positive semidefiniteness of~$\nabla_z^2 p$ almost everywhere, so that the remaining dual formulation of~\eqref{eq:laplacian_lifting_pd} reads
\begin{equation}
        \sup_{p \in C_c^2(\Omega \times \Gamma)} \int_{\Omega \times \Gamma} (- \Delta_x p + \phi^{\lambda}) \d \mu \quad \text{s.t.} \ (-\nabla_z p, \phi^{\lambda}) \in \mathcal{K}_f \text{ and } \nabla_z^2 p \ \text{pos. semidef.}
        \label{eq:laplacian_lifting_d}
\end{equation}
Clearly, \eqref{eq:laplacian_lifting_d} equals~\eqref{eq:laplace_lifting_ssvm} up to the sign of the variable~$p$ -- in fact, this is only a notational deviation since the positive semidefiniteness of~$\nabla_z^2 p$ is equivalent to the concavity of~$-p$.
At the same time, note that this does not prove the missing (in-)equality in~\eqref{eq:laplace_lifting_ssvm_ineq} as the derivation of~\eqref{eq:laplacian_lifting_d} is only formal and hard to justify rigorously.
Therefore, the compatibility of a Laplacian lifting with the original problem (as in \eqref{eq:laplacian_lifting_compatibility}) remains exclusive to the primal formulation~\eqref{eq:laplacian_lifting}.

\subsubsection{Dynamical Optimal Transport and Harmonic Mappings \cite{brenier2003,lavenant2019}}
As alluded to above, the proposed Benamou-Brenier functional~\eqref{eq:bb_functional} can be seen as a direct extension of the target functional~\eqref{eq:benamou_brenier} for dynamical transport problems.
By the well-known conjugacy relationship
\begin{equation}
        \frac{ \| \cdot \|^p }{p} \overset{*}{\longleftrightarrow} \frac{ \| \cdot \|^q }{q} \quad \text{for} \ p, q \in (1, \infty) \ \text{with} \ \frac{1}{p} + \frac{1}{q} = 1,
\end{equation}
it is easy to see that the dually admissible set~$\mathcal{K}_q$ from~\eqref{eq:benamou_brenier_dual_set} is a special case of the admissible set $\mathcal{K}_f$ from~\eqref{eq:bb_dual_set} for integrands of the form $f(t, v) = \| v \|^p / p$.
Apart from the lack of a boundary condition on~$\mu$ as in \eqref{eq:benamou_brenier}, the transport problem~\eqref{eq:benamou_brenier} is therefore equivalent to our proposed Benamou-Brenier functional~\eqref{eq:bb_functional} for such integrands under the continuity equation from~\eqref{eq:benamou_brenier}.

\begin{figure}[t]
	\centering
	\input{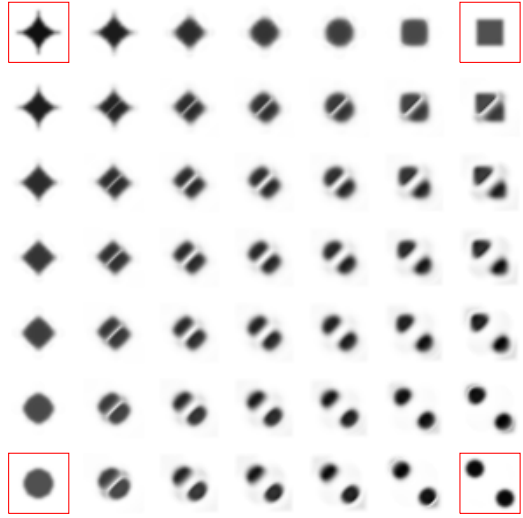}
	\caption{
		Illustration of a harmonic (probability) measure-valued mapping~${\mu : \Omega \to \mathcal{P}(\Gamma)}$ for two-dimensional~$\Omega$ and~$\Gamma$:
		Given the values of~$\mu$ at the four corners of a rectangular region~$\Omega$ (marked with red boxes), a one-dimensional problem of the form~\eqref{eq:harmonic_wasserstein_problem} is solved for each of the four boundaries.
		In a second step, these interpolants serve as boundary values for a two-dimensional problem~\eqref{eq:harmonic_wasserstein_problem} that generates the displayed interpolation in the interior of $\Omega$.
		Illustration inspired by \cite{solomon2015}.
	}
	\label{fig:harmonic_interpolation_2d}
\end{figure}

On an abstract level, this means that while dynamical optimal transport solutions can be seen as \textit{interpolants} between measures \textit{given on the spatial boundary}~$\partial \Omega$ satisfying a predefined notion of regularity, our method generates \textit{minimal energy surfaces} inside of cost landscapes \textit{given over the spatial region}~$\Omega$ which adhere to the same notion of regularity.
Said cost landscapes are thereby represented by a data term~$\rho$ that amends the above-discussed regularization term in the integrand~$f$.

One special case of optimal transport interpolants which was studied in~\cite{brenier2003,lavenant2019} are mappings~$\mu$ from a (possibly multidimensional) region~$\Omega$ to the space of probability measures~$\mathcal{P}(\Gamma)$ that minimize the energy corresponding to the integrand~$f(t, v) = \| v \|^2 / 2$ under the continuity equation~\eqref{eq:1st_order_ce} as well as prescribed data $\mu_b$ on the spatial boundary~$\partial \Omega$.
More precisely, consider
\begin{multline}
        \inf_{(\mu, E)} \sup_{(\alpha, \beta)} \bigg\{ \int_{\Omega \times \Gamma} \langle (\alpha, \beta), \d (\mu, E) \rangle : \alpha(t) + \frac{\|\beta(t)\|^2}{2} \leq 0 \ \forall t \in \Omega \times \Gamma \bigg\} \\
        \text{s.t.} \ \nabla_x \mu + \ddiv_z E = 0, \ \mu = \mu_b \ \text{on} \ \partial \Omega.
        \label{eq:harmonic_wasserstein_problem}
\end{multline}
Such mappings can be seen to generalize the concept of harmonic functions (cf.~\cite[Section~3]{brenier2003}) and are  termed \textit{harmonic mappings with values in the Wasserstein space}.
See Figure~\ref{fig:harmonic_interpolation_2d} for an illustration of such a mapping.

Further discussed in~\cite{brenier2003} is a system of (formally derived) optimality equations for solutions~$(\mu, E)$ to problem~\eqref{eq:harmonic_wasserstein_problem}.
As we shall show, these equations can be linked to non-trivial solutions of our proposed Laplacian continuity equation~\eqref{eq:laplacian_ce}.

\begin{prop}
        Let $(E, \mu) \in \mathcal{M}(\Omega \times \Gamma, \R^{d, s} \times \R)$ be a solution of~\eqref{eq:1st_order_ce} with a Radon-Nikodym density $\d E / \d \mu =: e = (e_1, \ldots, e_s) \in L^1_{\mu}(\Omega \times \Gamma, \R^{d, s})$ that additionally satisfies
        \begin{equation}
                \int_{\Omega \times \Gamma} \langle \nabla_x \phi, \d E \rangle + \int_{\Omega \times \Gamma} \sum\nolimits_{i = 1}^s \langle \nabla_z \phi, e_i \otimes e_i \rangle \d \mu = 0
                \label{eq:harmonic_wasserstein_optimality}
        \end{equation}
        for all $\phi \in C_c^1(\Omega \times \Gamma, \R^s)$, i.e., that solves equation~(57) from~\cite{brenier2003}.
        Then, the triple~$((\sum_i e_i \otimes e_i) \mu, 0, \mu) \in \mathcal{M}(\Omega \times \Gamma, \R^{s, s} \times \R^s \times \R)$ is a solution of the Laplacian continuity equation~\eqref{eq:laplacian_ce}.
        \label{prop:harmonic_laplacian_ce}
\end{prop}
\begin{proof}
        Let $\psi \in C_c^2(\Omega \times \Gamma, \R)$ be arbitrary.
        Then, one has $\nabla_x \psi \in C_c^1(\Omega \times \Gamma, \R^d)$ as well as $\nabla_z \psi \in C_c^1(\Omega \times \Gamma, \R^s)$.
        Therefore, it holds by~\eqref{eq:1st_order_ce} that
        \begin{equation}
                \int_{\Omega \times \Gamma} \ddiv_x (\nabla_x \psi) \d \mu + \int_{\Omega \times \Gamma} \langle \nabla_z (\nabla_x \psi), \d E \rangle = 0
                \label{eq:harmonic_laplacian_ce_eq1}
        \end{equation}
        and, by~\eqref{eq:harmonic_wasserstein_optimality}, that
        \begin{equation}
                \int_{\Omega \times \Gamma} \langle \nabla_x (\nabla_z \psi), \d E \rangle + \int_{\Omega \times \Gamma} \sum\nolimits_{i=1}^s \langle \nabla_z (\nabla_z \psi), e_i \otimes e_i \rangle \d \mu = 0.
                \label{eq:harmonic_laplacian_ce_eq2}
        \end{equation}
        Substituting~\eqref{eq:harmonic_laplacian_ce_eq1} into~\eqref{eq:harmonic_laplacian_ce_eq2} yields
        \begin{equation}
                -\int_{\Omega \times \Gamma} \Delta_x \psi \d \mu + \int_{\Omega \times \Gamma} \left\langle \nabla_z^2 \psi, \sum\nolimits_{i=1}^s e_i \otimes e_i \right\rangle \d \mu = 0,
        \end{equation}
        so that $((\sum_i e_i \otimes e_i) \mu, 0, \mu)$ solves~\eqref{eq:laplacian_ce} by Definition~\ref{def:laplacian_ce}. \qed
\end{proof}
 
Note that the vanishing of the second component of~$((\sum_i e_i \otimes e_i) \mu, 0, \mu)$ is somewhat natural as the optimality condition~\eqref{eq:harmonic_wasserstein_optimality} was derived for \textit{harmonic mappings} and as this component was shown to ``correspond'' to the Laplacian~$\Delta u$ for solutions $\mu$ concentrated on the graph of $u \in C^2(\Omega, \Gamma)$ in Propositions~\ref{prop:laplacian_ce_nc} and~\ref{prop:laplacian_ce_sc}.
Put differently, this consideration shows the consistency of the proposed Laplacian continuity equation~\eqref{eq:laplacian_ce} with previous findings on harmonic mappings in the Wasserstein space.

\section{Discretization \& Numerical Experiments}
\label{sec:results}

In this section, we present a number of numerical experiments on standard imaging problems to demonstrate the functionality of the lifting models proposed above.
All of the following experiments were conducted using the primal-dual hybrid gradient optimization algorithm from \cite{pock2009,chambolle2011} and, more specifically, its GPU-based implementation in the \verb|prost|-library\footnote{\url{https://github.com/tum-vision/prost}} as well as its \verb|sublabel_relax|-extension\footnote{\url{https://github.com/tum-vision/sublabel_relax}} on a machine with an Intel i7-8700 CPU, 64~GB of main memory and a NVIDIA GeForce RTX~2070 GPU featuring 8~GB of video memory.

Backprojection of the results (``unlifting'') was achieved by a simple averaging procedure, i.e., a function~$u$ is obtained from a minimizer~$\mu$ of a lifted functional~$\mathcal{F}$ by computing the expectation $u(x) := \int_{\Gamma} z \d \mu((x, z))$ at every $x \in \Omega$.

\subsection{First-Order Model}
\label{subsec:first_order_experiments}

This section is concerned with a discretized version of our first-order lifting model~\eqref{eq:1st_order_lifting}, in which we employ the following variant of a total variation regularizer from~\cite{lellmann2013b}:
\begin{equation}
\TV(u) := \sup \left\{ \int_{\Omega} \langle u, \ddiv \phi \rangle \d x \ : \ \phi \in C_c^1(\Omega, \R^{d, s}), \ \| \phi \|_{\sigma} \leq 1 \right\},
\label{eq:tv_def}
\end{equation}
where $\| \cdot \|_{\sigma}$ is the spectral norm.
Consequently, one has
\begin{equation}
\TV(u) = \int_{\Omega} \| \nabla u \|_* \d x
\label{eq:tv_differentiable}
\end{equation}
for~${u \in W^{1,1}(\Omega, \R^s)}$ with $\|\cdot\|_*$ as the nuclear norm, i.e., the sum of the arguments' singular values.
Note that~\eqref{eq:tv_differentiable} justifies the application of our first-order model in the following experiments.

Total variation regularization is known to be computationally favorable for convex relaxation models involving constraint sets akin to $\mathcal{K}_f$ from \eqref{eq:bb_dual_set} as it decouples the constraint $f^*(t, \phi^{\xi}(t)) + \phi^{\lambda}(t) \leq 0$ for $f(t, p) = \rho(t) + \| p \|_*$ into separate constraints for data term and regularizer, namely
\begin{equation}
\phi^{\lambda}(t) \leq \rho(t), \quad \| \phi^{\xi}(t) \|_{\sigma} \leq 1.
\end{equation}

Hence, a straightforward discretization of~\eqref{eq:1st_order_lifting} for TV-regularized models is given by
\begin{align}
	\inf_{\mu, E} \ \sup_{\substack{\phi^{\lambda}, \phi^{\xi}\\q}} \ &\sum_{t \in \Omega_h \times \Gamma_h} \mu(t)  \left( \phi^{\lambda}(t) + \ddiv^h_x q(t) \right) + \left\langle E(t), \phi^{\xi}(t) + \nabla^h_z q(t) \right\rangle \label{eq:tv_model_disc}\\
	\text{s.t.} \quad & \phi^{\lambda}(t) \leq \rho(t), \quad \| \phi^{\xi}(t) \|_{\sigma} \leq 1, \quad \mu(t) \geq 0 \qquad \forall t \in \Omega_h \times \Gamma_h, \nonumber\\
	& \sum\nolimits_{z \in \Gamma_h} \mu((x, z)) = 1 \qquad \forall x \in \Omega_h \nonumber
\end{align}
with suitable finite grids $\Omega_h \subset \Omega$ and $\Gamma_h \subset \Gamma$ discretizing the problems' domain and range.
Accordingly, $\ddiv_x^h$ and $\nabla_z^h$ denote finite difference approximations to the respective differential operators and are implemented using Neumann boundary conditions.

Although our focus lies neither on competitive performance nor efficiency but rather on a proof of concept, we want to shortly address the issue of memory requirements:
As all of the primal and dual variables involved in \eqref{eq:tv_model_disc} are defined over~${\Omega_h \times \Gamma_h}$, the models' memory consumption scales with the overall number of discretization points $|\Omega_h| \cdot |\Gamma_h|$.
Unfortunately, this issue already limits the models' applicability for input images of moderate resolution $|\Omega_h|$ when accurate solutions are sought, i.e., when large numbers $|\Gamma_h|$ of range discretization points, so-called \textit{labels}, are employed.

As a remedy, a more sophisticated \textit{sublabel-accurate} discretization strategy for lifting models with TV-regularization was developed in \cite{mollenhoff2016,laude2016}. This strategy allows for a more accurate discretization of the data term for small $|\Gamma_h|$.
Since a sublabel-accurate formulation of~\eqref{eq:1st_order_lifting} is notationally involved and at the same time largely equivalent to the presentation in~\cite{laude2016}, we omit it here and refer to the above-mentioned publications for details.

Instead, we resort to the intuitive explanation given in~\cite[Proposition~4]{mollenhoff2017} that a sublabel-accurate discretization corresponds to an approximation of the dual variables $\phi^{\lambda}$ by finite elements of first order.
A straightforward discretization as in \eqref{eq:tv_model_disc} on the other hand can be linked to an approximation by elements of zeroth order~\cite[Proposition~2]{mollenhoff2017}.
For all details, we refer to the named publications.

\subsubsection{Stereo Matching}

\begin{figure}[t]
	\resizebox{\textwidth}{!}{
		\input{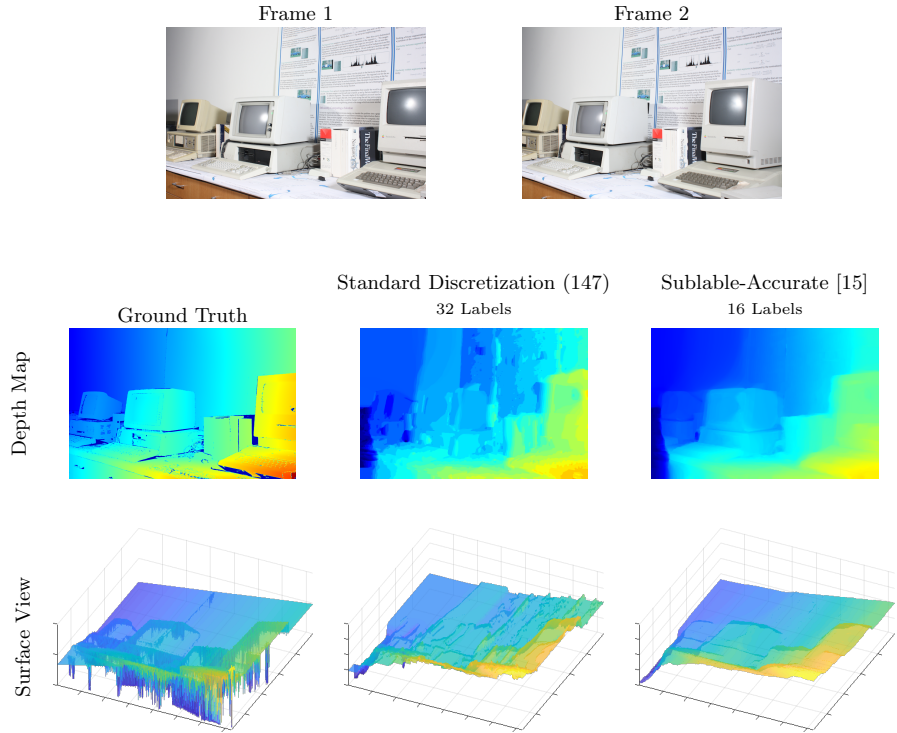}
	}
	\caption{
		Results of the proposed first-order model \eqref{eq:1st_order_lifting} on the scalar-valued problem of finding stereo correspondences between two rectified views of the same scenery.
		As the results in the figures' lower half indicate, the proposed model exhibits a basic functionality when discretized as in \eqref{eq:tv_model_disc} (\textbf{center column}).
		Said approach, however, is prone to label bias as seen from the prominent staircase nature of the displayed depth map.
		A more sophisticated, sublabel-accurate discretization (\textbf{right column}) as in~\cite{mollenhoff2016,laude2016,mollenhoff2017} resolves this problem while at the same time requiring much fewer labels.
		Image data: stereo pair ``Vintage'' from \cite{scharstein2014}.
	}
	\label{fig:stereo_vintage}
\end{figure}

As a nonconvex scalar problem, we test our model on the task of estimating stereo correspondences between two rectified views~$I_1$ and~$I_2$ of the same scenery, i.e., for every point $x \in \Omega$, a horizontal displacement~$u(x)$ is sought that matches $I_1((x_1, x_2 + u(x)))$ with $I_2((x_1, x_2))$.

For our experiments on the test data from \cite{scharstein2014}, we use the upper bounds on the maximum displacement provided by the authors to estimate a suitable search range~$\Gamma$.
Alongside the above-discussed TV-regularization, we use the stereo matching data term implemented in \verb|prost|, i.e.,
\begin{multline}
	\rho((x, u(x))) := \int_{W(x)} h(\partial_{x_1} I_1((y_1, y_2 + u(x))) - \partial_{x_1} I_2(y)) \\
	+ h(\partial_{x_2} I_1((y_1, y_2 + u(x))) - \partial_{x_2} I_2(y)) \d y
\end{multline}
with $h(\alpha) := \min\{ |\alpha|, \nu \}$ for a suitable threshold $\nu > 0$ and with averaging windows~${W(x) \subset \Omega}$.

Experimental results of our model on the ``Vintage'' image pair from \cite{scharstein2014} can be seen in Figure~\ref{fig:stereo_vintage}. Using the standard discretization scheme \eqref{eq:tv_model_disc}, the model shows basic functionality although the results are evidently prone to label bias. More accurate results with fewer labels are achieved using the above mentioned sublabel-accurate discretization scheme. 

\subsubsection{Optical Flow}

\begin{figure}[t]
	\resizebox{\textwidth}{!}{
		\input{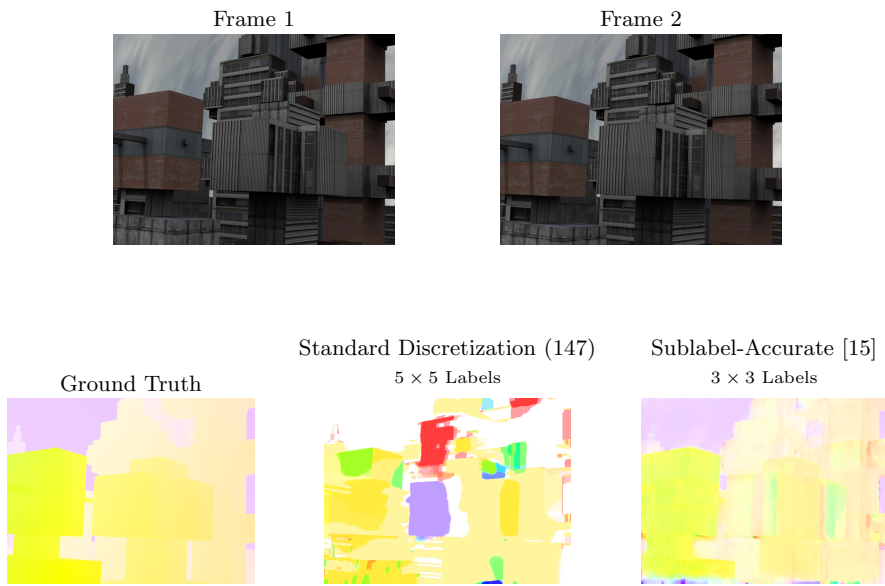}
	}
	\caption{
		Results of the proposed first order model~\eqref{eq:1st_order_lifting} for the vector-valued problem of estimating the optical flow between two consecutive frames of a video sequence.
		As in Figure~\ref{fig:stereo_vintage}, a standard discretization (\textbf{lower half, center}) merely exhibits bare functionality when using reasonable numbers of labels.
		In order to obtain acceptable solutions, a refined sublabel-accurate discretization (\textbf{lower half, right}) needs to be employed.
		Note that an efficient discretization is all the more crucial for vectorial problems as the limiting factor of (graphics) memory requirements scales proportionally with the number of labels used \textit{in each dimension of the range}~$\Gamma$.
		Image data: sequence ``Urban3'' from \cite{baker2010}.
	}
	\label{fig:of_urban3}
\end{figure}

As a test problem for the vector-valued version of \eqref{eq:tv_model_disc}, we consider the task of optical flow estimation.
Given two consecutive frames~$I_1$ and~$I_2$ from a video sequence, one seeks to find a vector field~$u$ explaining the physical motion between the two.

The data term for our experiments is given by a simple $L^1$-distance
\begin{equation}
	\rho((x, u(x))) = \| I_1(x + u(x)) - I_2(x) \|_2.
	\label{eq:of_l1_distance}
\end{equation}
Again, we inferred the size of a suitable search window~$\Gamma$ from the data provided by the authors of the dataset~\cite{baker2010}.

Figure~\ref{fig:of_urban3} shows the results of our model on the sequence ``Urban3''. For vectorial problems  -- such as optical flow estimation -- memory consumption scales with the amount of labels employed per dimension. Therefore, the practical number of labels is even more restricted then for scalar problems and an efficient discretization is all the more crucial. Accordingly in Figure~\ref{fig:of_urban3} the benefit of using sublabel-accurate discretization instead of the straightforward one~\eqref{eq:tv_model_disc} is even more apparent.

\subsection{Laplacian Model}
\label{subsec:laplacian_experiments}

\begin{figure}[t]
	\resizebox{\textwidth}{!}{
		\input{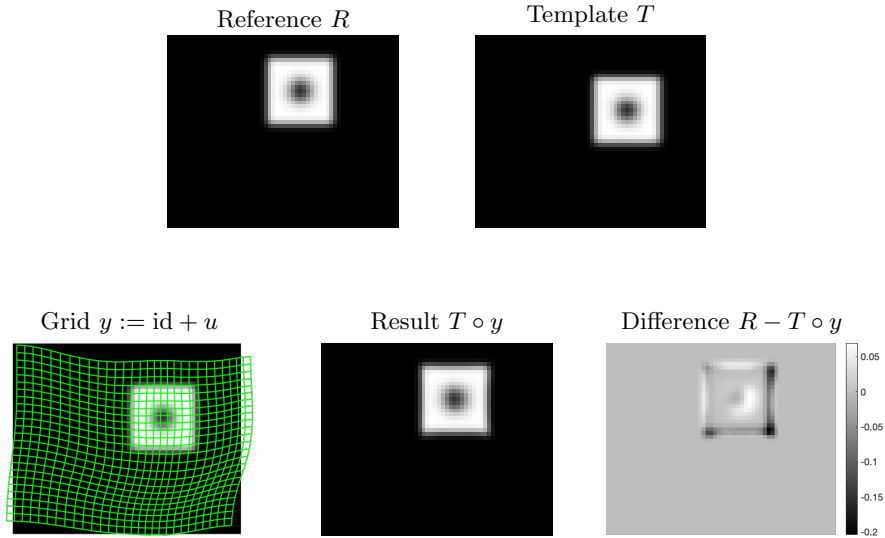}
	}
	\caption{
		Results of the proposed Laplacian lifting model~\eqref{eq:laplacian_lifting} for the problem of registering two synthetic images, i.e., a deformation~$u$ is sought that aligns the template image~$T$ (\textbf{upper half, right}) with the reference image~$R$ (\textbf{upper half, left}).
		The Laplacian model penalizes the curvature of the deformation and is commonly used for (medical) image registration~\cite{fischer2003}.
		In order to obtain a reasonably accurate deformation in spite of a non-sublabel-accurate discretization, we employed a search window of~$[-8, 8] \times [-8, 8]$ for the range $\Gamma$ -- the ground truth displacement is a translation by the vector $(5, 5)$ -- and sampled it using~$31$ equidistant labels in each direction.
		As the registration result (\textbf{lower half, center}) and the difference image (\textbf{lower half, right}) show, the method is able to successfully register~$T$ onto~$R$ through a smooth deformation (\textbf{lower half, left}).
		Note however, that for pratical use on real-world data, a more efficient discretization is imperative due to the otherwise impraticable memory requirements.
	}
	\label{fig:ir_textured_rect}
\end{figure}

This section is concerned with a discretized version of the Laplacian lifting model~\eqref{eq:laplacian_lifting}.
Our main motivation for this model is, as in~\cite{loewenhauser2018,vogt2019a}, to solve image registration problems with curvature regularization~\cite{fischer2003} by $\frac{1}{2}\| \Delta u(x) \|_2^2$.

Although the task of registering a template image~$T$ onto a (similar) reference image~$R$ by the means of a deformation field~$u$ is conceptually equivalent to the above-described optical flow estimation problem, in case of image registration one is often interested in smooth deformations~$u$.
This is especially true in the context of medical image data where piecewise constant deformations -- as favored by a TV-regularizer -- are physically implausible.
On the contrary, a higher order penalization often provides the desired smoothness.

In order to formulate a computationally tractable discretization of~\eqref{eq:laplacian_lifting}, we still need to address the issue of implementing the dual admissibility constraint given in \eqref{eq:bb_dual_set}:
Since the function $\eta = \frac{1}{2} \| \cdot \|_2^2$ penalizing $\Delta u(x)$ is not one-homogenous, its conjugate does not evaluate to an indicator function and the constraint does not decouple with respect to data term and regularizer as it was the case in Section~\ref{subsec:first_order_experiments}.
Rather, one has $\eta^* = \frac{1}{2} \| \cdot \|_2^2$, so that the constraint reads
\begin{equation}
	- \rho(t) + \frac{1}{2} \| \phi^{\xi}(t) \|_2^2 + \phi^{\lambda}(t) \leq 0.
	\label{eq:laplacian_dual_constraint}
\end{equation}

As previously observed in \cite{laude2016,mollenhoff2017}, \eqref{eq:laplacian_dual_constraint} can be rephrased by making use of the fact, that the epigraph of $(g_1 + g_2)^*$ is the Minkowski sum of the respective epigraphs of $g_1^*$ and $g_2^*$ for proper, convex and lower-semicontinuous $g_1$ and $g_2$~\cite[Exercise~1.28]{rockafellar2004}.
Therefore, one obtains the following system of constraints
\begin{align}
	-\rho(t) + \phi^{\lambda_1}(t) &\leq 0,\\
	\frac{1}{2} \| \phi^{\xi}(t) \|_2^2 + \phi^{\lambda_2}(t) & \leq 0,\\
	\phi^{\lambda_1}(t) + \phi^{\lambda_2}(t) &= \phi^{\lambda}(t),
\end{align}
which decouples the epigraphical constraints for data term and regularizer at the expense of an additional equality constraint.

With this in mind, we formulate a discretized version of~\eqref{eq:laplacian_lifting} as follows:
\begin{align}
	\inf_{\substack{\mu, E, H\\r}} \ \sup_{\substack{\phi^{\xi}, \phi^{\lambda}\\\phi^{\lambda_1}, \phi^{\lambda_2}\\q}} \ & \sum_{t \in \Omega_h \times \Gamma_h}
	\begin{Bmatrix*}[l]
	\mu(t) \left( \phi^{\lambda}(t) - \Delta_x^h q(t) \right)\\
	\quad + \left\langle E(t), \phi^{\xi}(t) + \nabla_z^h q(t) \right\rangle\\
	\qquad + \langle H(t), (\nabla_z^2)^h q(t) \rangle\\
	\qquad \quad + r(t) \left( \phi^{\lambda}(t) - \phi^{\lambda_1}(t) - \phi^{\lambda_2}(t) \right)
	\end{Bmatrix*}\\
	\text{s.t.} \quad & \left. \begin{array}{l l}
	\rho(t) \geq \phi^{\lambda_1}(t), & \quad \frac{1}{2} \| \phi^{\xi}(t) \|_2^2 + \phi^{\lambda_2}(t) \leq 0,\\
	\mu(t) \geq 0, & \quad H(t) \in S_+^s
	\end{array} \right\} \quad \forall t \in \Omega_h \times \Gamma_h \nonumber \\
	& \sum\nolimits_{z \in \Gamma_h} \mu((x, z)) = 1 \qquad \forall x \in \Omega_h, \nonumber
\end{align}
where $S_+^s$ is the cone of symmetric and positive semidefinite $s \times s$-matrices and where~$\Omega$ and~$\Gamma$ are discretized as in~\eqref{eq:tv_model_disc}.
As before, the differential operators $\Delta_x^h$, $\nabla_z^h$ and $(\nabla_z^2)^h$ are implemented using finite differences with Neumann boundary conditions.

Since a discussion of a sublabel-accurate discretization of~\eqref{eq:laplacian_lifting} is more involved due to the above-discussed issue of decoupling data term and regularizer constraints, we leave it for future work.
Instead, we present the results of a registration experiment performed on synthetic data in Figure~\ref{fig:ir_textured_rect}.
As a data term, we used the same $L^1$-distance~\eqref{eq:of_l1_distance} as in the optical flow experiments.

Figure~\ref{fig:ir_textured_rect} shows that the proposed model was not only able to successfully register the template image~$T$ onto the reference~$R$, but also, more interestingly, returned a smooth deformation.
\section{Conclusion}
\label{sec:conclusion}

In this work, we have presented a mathematically rigorous framework for functional lifting based on the theory of dynamical optimal transport. Said connection is established through a generalized Benamou-Brenier functional $\mathfrak{B}_f$. As the main theoretical contribution, we have proven an integral representation of $\mathfrak{B}_f$ in Theorem~\ref{thm:bb_integral_representation}. This concept allows to rephrase a large class of nonconvex variational problems as optimization problems over a convex functional.

The proposed framework can be seen as a direct generalization of the classic scalar-valued lifting approach from~\cite{pock2008,pock2010} and, unlike the latter ones, extends naturally to vectorial problems.
Due to its modular structure, our framework allows for various regularizers and types of differential operators as we have demonstrated in Section~\ref{subsec:constraint_lifting} and as such encompasses numerous models investigated separately throughout recent years.
We hope that this work will serve as a blueprint for future developments in this direction.

While we were able to show in~\eqref{eq:1st_order_equiv},~\eqref{eq:2nd_order_equiv} and~\eqref{eq:laplacian_lifting_compatibility} that the respective liftings~$\mathcal{F}$ agree with the original functional~$F$ on graph-concentrated measures, open questions regarding the minimization of~$\mathcal{F}$ over the relaxed domain~$L_w^{\infty}(\Omega, \mathcal{P}(\Gamma))$ include the following:
Can every minimizer of $\mathcal{F}$ be linked to one (or multiple) minimizer(s) of $F$?
If so, can minimizers of $F$ be computed from minimizers of~$\mathcal{F}$ by a suitable projection technique?
Such an extension of the thresholding theorem~\cite[Theorem~3.1]{pock2010} might however be out of reach as indicated by the experiments in~\cite{vogt2019a}.

\vspace{1em}

\noindent\textbf{Acknowledgments.}
The authors acknowledge support through DFG grant LE \mbox{4064/1-1}
``Functional Lifting 2.0: Efficient Convexifications for Imaging and Vision''
and NVIDIA Corporation.

\bibliographystyle{splncs04}
\bibliography{references}

\end{document}